\documentclass{article}[11pt]
\usepackage{a4}
\usepackage[english]{babel}
\usepackage[T1]{fontenc}
\usepackage[all]{xy}
\usepackage{amssymb}
\usepackage{amsmath}
\usepackage{amsthm}
\usepackage{graphicx}
\usepackage{caption}
\usepackage{rotating}
\usepackage{latexsym}
\usepackage{float}

\usepackage{tikz}
\usetikzlibrary{matrix}
\usetikzlibrary{shapes}
\usetikzlibrary{arrows}  
\usetikzlibrary{calc,3d}
\usetikzlibrary{decorations,decorations.pathmorphing}
\usetikzlibrary{through}
\tikzset{ext/.style={circle, draw,inner sep=1pt},int/.style={circle,draw,fill,inner sep=1pt},nil/.style={inner sep=1pt}}
\tikzset{exte/.style={circle, draw,inner sep=3pt},inte/.style={circle,draw,fill,inner sep=3pt}}
\tikzset{diagram/.style={matrix of math nodes, row sep=3em, column sep=2.5em, text height=1.5ex, text depth=0.25ex}}
\tikzset{diagram2/.style={matrix of math nodes, row sep=0.5em, column sep=0.5em, text height=1.5ex, text depth=0.25ex}}

\usepackage{scalerel,stackengine}
\stackMath
\newcommand\reallywidehat[1]{%
\savestack{\tmpbox}{\stretchto{%
  \scaleto{%
    \scalerel*[\widthof{\ensuremath{#1}}]{\kern-.0pt\bigwedge\kern-.0pt}%
    {\rule[-\textheight/2]{1ex}{\textheight}}
  }{\textheight}%
}{1.2ex}}%
\stackon[1pt]{#1}{\tmpbox}%
}
\newtheorem{thm}{Theorem}[section]
\newtheorem{prop}[thm]{Proposition}
\newtheorem{lem}[thm]{Lemma}
\newtheorem{coro}[thm]{Corollary}
\newtheorem{defn}[thm]{Definition}

\newtheorem{rmq}[thm]{Remark} 

\newtheorem{conj}[thm]{Conjecture}
\numberwithin{equation}{section}   

\newcommand{\rdbb}{\mathbb{R}^d}
\newcommand{\lra}{\longrightarrow}

\newcommand{\rbb}{\mathbb{R}}

\newcommand{\embthm}{\overline{\mathrm{Emb}}_c(\coprod_{i=1}^r \mathbb{R}^{m_i}, \rdbb)}
\newcommand{\emb}{\overline{\mathrm{Emb}}_c(\coprod_{i=1}^r \mathbb{R}^{m_i}, \rdbb)}
\newcommand{\embnobar}{\mathrm{Emb}_c(\coprod_{i=1}^r \mathbb{R}^{m_i}, \rdbb)}
\newcommand{\embsphere}{\mathrm{Emb}(\coprod_{i=1}^r S^{m_i}, S^d)}
\newcommand{\rmodthm}{\underset{\Omega}{\mathrm{hRmod}}}
\newcommand{\rmod}{\underset{\Omega}{\mbox{hRmod}}}
\newcommand{\qbb}{\mathbb{Q}}
\newcommand{\qms}{Q^{m_1\cdots m_r}_{s_1 \cdots s_r}}
\newcommand{\codim}{d > 2 \mathrm{max}\{m_i| \ 1 \leq i \leq r\} + 1}

\newcommand{\lie}{{{\mathcal L}ie}((\bullet))}

\newcommand{\emh}{\mathcal{E}_{H}^{m_1, \cdots, m_r; d}}
\newcommand{\empi}{\mathcal{E}_{\pi}^{m_1, \cdots, m_r; d}}
\newcommand{\mg}{\mathcal{G}}
\newcommand{\dhat}{\widehat{D}_d}
\newcommand{\phat}{\widehat{P}_d}

\newcommand{\codimlow}{d -\mathrm{max}\{m_i| \ 1 \leq i \leq r\} >2}

\newcommand{\MODL}{\mathbf{{Mod}}({\mathcal L}_\infty)}

\newcommand{\bigpibar}{\overline{\Pi}^{m_1\ldots m_r d}}
\newcommand{\biglbar}{\overline{\Lambda}^{m_1\ldots m_r d}}
\newcommand{\bigkbar}{\overline{K}^{m_1\ldots m_r d}}
\newcommand{\bigpi}{\Pi^{m_1\ldots m_r d}}
\newcommand{\bigla}{\Lambda^{m_1\ldots m_r d}}
\newcommand{\bigk}{K^{m_1\ldots m_r d}}

\newcommand{\Det}{\mathrm{Det}}

\newcommand{\MOD}{\mathbf{Mod}}
\newcommand{\MODLdet}{\mathbf{{Mod}}_{\Det}({\mathcal L}_\infty)}

\newcommand{\FCOM}{\mathrm{F}{\mathcal C}om}
\newcommand{\FCOMdet}{\mathrm{F}_\Det{\mathcal C}om}

\newcommand{\lieminus}{{{\mathcal L}ie}(\bullet-1)}

\title{ \textbf{Rational homology and homotopy of \\ high dimensional string links}}
\date{}

\author{Paul Arnaud Songhafouo Tsopm\'en\'e\\
Victor Turchin\thanks{The second author is partially supported by the Simons Foundation \lq\lq{}Collaboration grant for mathematicians\rq\rq{}, award ID:~519474.}}

\begin{document}
\maketitle

\begin{abstract}  
Arone and the second author showed that when the dimensions are in the stable range, the rational homology and homotopy of the high dimensional anologues of spaces of long knots can be  calculated as the homology of a direct sum of finite graph-complexes that they described explicitly.  They also showed that these homology and homotopy groups can be interpreted as the 
higher order Hochschild homology also called Hochschild-Pirashvili homology. 
 In this paper, we generalize all these results to high dimensional analogues of spaces of string links. 

The methods of our paper are applicable in the range when the ambient dimension is at least twice the maximal dimension of a 
link component plus two, which in particular guarantees that the spaces under the study are connected. However, we conjecture  that our homotopy graph-complex computes the rational homotopy groups of links spaces always when codimension is greater than two, i.e. always when the Goodwillie-Weiss calculus is applicable. Using Haefliger\rq{}s approach to calculate the groups of isotopy classes of higher dimensional links,
we confirm our cojecture at the level of $\pi_0$.
\end{abstract}


\setcounter{section}{-1}

\section{Introduction}\label{s:intro}

In this paper we further develop the connection between the study of embedding spaces, operad theory, and graph-complexes. Our results are partially influenced by Bar Natan\rq{}s work~\cite{barnatan95}, where he describes  the space dual to weight systems (counting Vassiliev invariants) of string links in $\rbb^3$ as a space of unitrivalent graphs modulo $AS$ and $IHX$ relations, with univalent vertices being colored by the components of the links, see also~\cite[Section~5.10]{chm_duzh_most12} for more details. This space is built from the components of the modular envelop $\mathbf{Mod}({{\mathcal L}ie})$ of  the operad of Lie algebras~\cite{hin_vaintrob02}. One of our main results says that the rational homotopy of the space of higher-dimensional string links is described as the homology of a graph-complex similarly built from the components of the modular envelop of ${\mathcal L}_\infty$. 
  In particular  Bar Natan\rq{}s space of unitrivalent graphs with colored univalent vertices appears as the bottom degree homology of our complexes. The approach that we use is from our previous works~\cite{aro_tur12,aro_tur13,songhaf13}. It is applicable only when codimension is high enough to guarantee that the spaces in question are connected. However, we conjecture that the same result holds requiring only  codimension  $>2$. To justify this conjecture we express the rational Haefliger groups of isotopy classes of 
  higher-dimensional links~\cite{haefliger}  as spaces of unitravalent trees with leaves colored by the components of the links
  (similarly as Bar Natan\rq{}s space). 
   We show that this is exactly the degree zero homology  of the  corresponding homotopy
  graph-complexes. In particular unitrivalent graphs of loop order $\geq 1$  are also present in the homology of these graph-complexes but they appear in positive  degrees.
  
    In this paper we  produce two different graph-complexes computing the rational homotopy of spaces of string links. One of them is built from the components of $\mathbf{Mod}({\mathcal L}_\infty)$ and another one has building blocks desrcibed in terms of the homology of configuration spaces
of points in Euclidean spaces. Both complexes resemble those appearing in the Bott-Taubes type integration, 
see for example~\cite{CCL,sakai10,sakai_watanabe,volic}. The obtained complexes all split into a direct sum of finite complexes. In the follow up paper~\cite{songhaf_tur}, we compute the generating function of the Euler characteristics of the summands of these complexes 
(both for the homology and homotopy). As a byproduct result in~\cite{songhaf_tur}, we compute the supercharacter of the symmetric group action on the positive arity components of  $\mathbf{Mod}({\mathcal L}_\infty)$. 

Let $d$ be the dimension of the ambient space. Let $r \geq 1$, and let $m_1, \cdots, m_r \geq 1$ be $r$ integers. A \textit{high-dimensional string link of $r$ strands} is a smooth embedding $f \colon \coprod_{i=1}^r \rbb^{m_i} \hookrightarrow \rdbb$  that coincides outside a compact set with a fixed  embedding  $\iota\colon \coprod_{i=1}^r \rbb^{m_i} \hookrightarrow \rdbb$ affine on each component. When $r$ is understood we will just say \textit{high-dimensional string link}. The space of high-dimensional string links, denoted $\mbox{Emb}_c(\coprod_{i=1}^r \mathbb{R}^{m_i}, \rdbb)$, is the collection of all high-dimensional string links endowed with the weak $\mathcal{C}^{\infty}$-topology. Similarly, we define the space $\mbox{Imm}_c(\coprod_{i=1}^r \mathbb{R}^{m_i}, \rdbb)$ of smooth immersions $\coprod_{i=1}^r \rbb^{m_i} \looparrowright \rdbb$ that coincide outside a compact set with~$\iota$. By definition, there is an inclusion $\mbox{Emb}_c(\coprod_{i=1}^r \mathbb{R}^{m_i}, \rdbb) \hookrightarrow \mbox{Imm}_c(\coprod_{i=1}^r \mathbb{R}^{m_i}, \rdbb)$, and its homotopy fiber over $\iota$, denoted $\emb$, is the space under our study in this paper. The case $r=1$ was extensively studied in \cite{aro_tur12} and \cite{aro_tur13} by Arone and the second author. They  obtain the Hodge decomposition in the rational homology $H_*(\overline{\mbox{Emb}}_c(\rbb^m, \rdbb), \qbb)$ and in the rational homotopy $\qbb \otimes \pi_*\overline{\mbox{Emb}}_c(\rbb^m, \rdbb)$, $d> 2m+1$. Then they explicitly describe graph-complexes computing $H_*(\overline{\mbox{Emb}}_c(\rbb^m, \rdbb), \qbb)$ and $\qbb \otimes \pi_*\overline{\mbox{Emb}}_c(\rbb^m, \rdbb)$.

Because of the high codimension, the space $\emb$ is connected and moreover its homotopy type does not depend on the choice of the fixed affine embedding $\iota$. In particular we can choose $\iota$ in a way that  all connected components of the fixed unlink are parallel to a given line. Concatenation of links along the direction of this line defines an $H$-space structure on $\emb$, which implies that the rational homology of $\emb$ is a cofree cocommutative coalgebra cogenerated by the rational homotopy. The same is true for the graph-complexes that we  study.

In order to describe our graph-complexes, we need Theorem~\ref{rational_homology_thm} below, which is a generalization of \cite[Theorem 1.3]{aro_tur12}. Before stating Theorem~\ref{rational_homology_thm}, we will recall two notions. The first one is that of right $\Omega$-modules. Let $\Omega$ be the category of finite unpointed  sets whose morphisms are surjections. A \textit{right $\Omega$-module} is a contravariant functor from $\Omega$ to any given category. We denote by $\underset{\Omega}{\mbox{Rmod}}$ the category of right $\Omega$-modules   of  chain complexes over $\qbb$. Abusing notation let $i\in\Omega$ denote the set $\{1,\ldots,i\}$. The category of right $\Omega$-modules of $\qbb$-vector spaces is an abelian category with the set of projective generators $\qbb[\mathrm{hom}_\Omega(i,-)]$, $i\geq 0$. 
 Using this
fact we endow $\underset{\Omega}{\mbox{Rmod}}$ with the structure of a (cofibrantly generated) model category, whose weak equivalences are quasi-isomorphisms, and fibrations are degreewise surjective maps~\cite{hovey}. (We call it projective model 
structure.) We denote by 
$\underset{\Omega}{\mbox{Rmod}}(-,-)$ the space (chain complex) of maps between two $\Omega$-modules, and by 
$\rmod(-,-)$ the derived mapping space. For specific computations we will need to apply this construction only to $\Omega$-modules with zero differential, in which case $\rmod(-,-)$  can be expressed as a product of $\mathrm{Ext}$ groups. 

As an example of a right $\Omega$-module (the one that appears in Theorem~\ref{rational_homology_thm}), we have the following. Let $X$ be a pointed topological space. Define the functor $X^{\wedge \bullet}$ that assigns to the set $\{1,2,\ldots,k\}$ the space  $ X^{\wedge k} = \underbrace{X \wedge \cdots \wedge X}_k$ (here "$\wedge$" is the smash product operation).  By $X^{\wedge 0}$ we mean the two-point  space -- zero dimensional pointed sphere $S^0$. 
The morphisms are induced by the diagonal maps. This functor is a right $\Omega$-module with values in topological spaces. Thus the functors  $\widetilde{C}_*(X^{\wedge \bullet}) $ (here $\widetilde{C}_*(-)$ denotes the reduced chain complex functor), $\widetilde{H}_*(X^{\wedge \bullet})$ (with zero differential) are objects of $\underset{\Omega}{\mbox{Rmod}}$. 

  The second notion we need to recall is that of right $\Gamma$-modules, which is the same as the previous one with the category $\Omega$ replaced by the category $\Gamma$. Recall that objects in $\Gamma$ are finite pointed sets, while morphisms are pointed maps.  As before we have the category $\underset{\Gamma}{\mbox{Rmod}}$ (of right $\Gamma$-modules
in chain complexes). We similarly denote by $\underset{\Gamma}{\mbox{Rmod}}(-,-)$ the mapping space and by $\underset{\Gamma}{\mbox{hRmod}}(-,-)$ the derived mapping space between two right $\Gamma$-modules. The example of a right $\Gamma$-module we look at is $H_*(C (\bullet, \rdbb), \qbb)$, $d\geq 2$, where $C(k, \rdbb)$ denotes the configuration space of $k$ labeled points in $\rdbb$. A way to see that $H_*(C(\bullet, \rdbb), \qbb)$ is a right $\Gamma$-module is the fact that it is an infinitesimal bimodule (see \cite[Definition 3.8]{aro_tur12} or \cite[Definition 4.1]{turchin10} for the definition of an \textit{infinitesimal bimodule}) over ${{\mathcal C}om} = H_0(C(\bullet, \rdbb), \qbb)$ since there is an obvious morphism of operads ${{\mathcal C}om} \lra H_*(C(\bullet, \rdbb), \qbb)$, the target $ H_*(C(\bullet, \rdbb), \qbb) =  H_*({\mathcal B}_d(\bullet), \qbb)$ is the homology of the little discs operad ${\mathcal B}_d(\bullet)$, which is well known to be the graded Poisson algebras operad~\cite{Coh}. It is also well known that an infinitesimal bimodule over Com is the same thing as a right $\Gamma$-module (see \cite[Corollary~4.10]{aro_tur12}
or \cite[Lemma~4.3]{turchin10}). One can also show that the sequence  $\qbb \otimes \pi_*C(\bullet,\rdbb)$, $d\geq 3$,  has a natural structure of a right $\Gamma$-module. 

  In \cite{pira00} Pirashvili constructs a functor (called \textit{cross effect}) $\mbox{cr} \colon \underset{\Gamma}{\mbox{Rmod}} \lra \underset{\Omega}{\mbox{Rmod}}$ that turns out to be an equivalence of categories. 
    In Section~\ref{rational_homology_string_link_section} we recall its definition.  This functor allows to pass from $\Gamma$-modules to $\Omega$-modules, which reduces computations as $\Omega$ has smaller sets of morphisms. 
    Let $\widehat{H}_*(C(\bullet, \rdbb), \qbb)$ (respectively $\qbb \otimes \widehat{\pi}_*C(\bullet, \rdbb)$) denote the cross effect of $H_*(C(\bullet, \rdbb), \qbb)$ (respectively the cross effect of $\qbb \otimes \pi_*C(\bullet, \rdbb)$).  

We are now ready to state our first result, which concerns the rational homology of $\emb$. 

\begin{thm} \label{rational_homology_thm}

For $d > 2 \mathrm{max}\{m_i| \ 1 \leq i \leq r\} + 1$, there is an isomorphism

\begin{eqnarray} \label{rational_homology_iso}
H_*(\embthm, \qbb) \cong H\left(\rmodthm\left(\widetilde{H}_*((\vee_{i=1}^r S^{m_i})^{\wedge \bullet}, \qbb), \widehat{H}_*(C(\bullet, \rdbb), \qbb)\right)\right).
\end{eqnarray}

\end{thm}

\sloppy
%

Our second result concerns the rational homotopy of $\emb$, and is the following. 

\begin{thm} \label{rational_homotopy_thm}
For $d > 2 \mathrm{max}\{m_i| \ 1 \leq i \leq r\} + 1$, there is an isomorphism

\begin{eqnarray} \label{rational_homotopy_iso}
\qbb \otimes \pi_*\embthm \cong H\left(\rmodthm\left(\widetilde{H}_*((\vee_{i=1}^r S^{m_i})^{\wedge \bullet}, \qbb), \qbb \otimes \widehat{\pi}_*C(\bullet, \rdbb)\right)\right).
\end{eqnarray}

\end{thm}

The main ingredients in the proof of Theorem~\ref{rational_homotopy_thm} are Theorem~\ref{rational_homology_thm} and hairy graph-complexes that  will be explicitly described in Subsection~\ref{graphs_complexes_emh-empi}. 

Below we produce a decomposition in the homology $H_*(\emb, \qbb)$. Our decomposition can be viewed as a generalization of a result of Pirashvili \cite{pirash00}, which concerns the Hodge decomposition for the $S^m$-homology of any right $\Gamma$-module (a.k.a. Hodge decomposition for the higher order Hochschild homology). We say "generalization" because, as in \cite{songhaf13} or as we will see in Section~\ref{rational_homology_string_link_section}, it is not difficult   to show that the homology $H_*(\emb, \qbb)$ is the $\vee_{i=1}^r S^{m_i}$-homology of the $\Gamma$-module $H_*(C(\bullet, \qbb))$.  The Hodge-type splitting that we study can be in the same way defined for the $\vee_{i=1}^r S^{m_i}$ higher Hochschild homology of any $\Gamma$-module and appears because ${C}_*((\vee_{i=1}^r S^{m_i})^{\times \bullet}, \qbb)$ is formal as a right $\Gamma$-module (which itself is equivalent to the fact that its cross effect $\Omega$-module $\widetilde{C}_*((\vee_{i=1}^r S^{m_i})^{\wedge \bullet}, \qbb)$ is formal). More generally such splitting in higher Hochschild 
homology takes place for any suspension and has been recently studied in~\cite{tur_wilw2}. 

The right hand sides of (\ref{rational_homology_iso}) and (\ref{rational_homotopy_iso}) can be split as follows. For any homomorphism of finite groups $H \to G$, let $\mbox{Ind}^G_H$ denote the induction functor from $H$-representations to $G$-representations, which is left adjoint to the restriction functor.\footnote{Explicitly, since we deal with the right action: $\mathrm{Ind}^G_H V=V\otimes_{\qbb[H]}\qbb[G].$} In most of the cases $G$ will be a symmetric group and $H$ be its subgroup.  Let us consider the right $\Omega$-module $\widetilde{H}_*((\vee_{i=1}^r S^{m_i})^{\wedge \bullet}, \qbb)$ that appears in (\ref{rational_homology_iso}). By definition, its $k$th component is given by 
\begin{eqnarray} \label{homology_wedge}
  \widetilde{H}_*((\vee_{i=1}^r S^{m_i})^{\wedge k}, \qbb)   \cong _{\Sigma_k}
  \bigoplus_{s_1+\cdots+s_r =k} \mbox{Ind}^{\Sigma_k}_{\Sigma_{s_1} \times \cdots \times \Sigma_{s_r}} 
    \widetilde{H}_*\left( S^{s_1m_1 + \cdots+s_rm_r}, \qbb    \right). 
\end{eqnarray}
One has the following isomorphism of right $\Omega$-modules

\begin{eqnarray} \label{splitting_htilde}
\widetilde{H}_*((\vee_{i=1}^r S^{m_i})^{\wedge \bullet}, \qbb) \cong \underset{s_1, \cdots, s_r \geq 0}{\bigoplus} Q^{m_1 \cdots m_r}_{s_1 \cdots s_r}, 
\end{eqnarray}

where $Q^{m_1 \cdots m_r}_{s_1 \cdots s_r}$ is the right $\Omega$-module defined by 

   \begin{eqnarray} \label{qs1sr}
Q_{s_1 \cdots s_r}^{m_1 \cdots m_r}(k) = \left\{ \begin{array}{lll}
                                 0 & \mbox{if} & k \neq s_1 + \cdots +s_r; \\
																\mbox{Ind}^{\Sigma_k}_{\Sigma_{s_1} \times \cdots \times \Sigma_{s_r}} \widetilde{H}_*(S^{s_1m_1+\cdots+s_rm_r} ; \mathbb{Q}) & \mbox{if} & k = s_1+\cdots+s_r.
                                \end{array} \right.
   \end{eqnarray}
Any non-bijective surjection acts on $Q^{m_1 \cdots m_r}_{s_1 \cdots s_r}$ automatically as zero. 
		
On the other hand, we have the splittings of $\Omega$-modules
\begin{eqnarray} \label{splitting_hhat_H}
\widehat{H}_*(C(\bullet, \rdbb), \qbb) = \underset{t \geq 0}{\prod} \widehat{H}_{t(d-1)}(C(\bullet, \rdbb), \qbb). 
\end{eqnarray}

\begin{eqnarray} \label{splitting_hhat_Pi}
\qbb \otimes \widehat{\pi}_*C(\bullet, \rdbb) = \underset{t \geq 0}{\prod} \qbb \otimes \widehat{\pi}_{t(d-2)+1}C(\bullet, \rdbb). 
\end{eqnarray}

Combining (\ref{rational_homology_iso}), (\ref{rational_homotopy_iso}), (\ref{splitting_htilde}), (\ref{splitting_hhat_H}), and (\ref{splitting_hhat_Pi}), it is not difficult to get the following splittings 

\begin{multline} \label{final_splitting_H}
H_*(\emb, \qbb) \cong \underset{s_1, \cdots, s_r, t}{\prod} H_*\rmod \left(Q^{m_1 \cdots m_r}_{s_1 \cdots s_r},  \widehat{H}_{t(d-1)}(C(\bullet, \rdbb), \qbb) \right)\\
\cong \underset{s_1, \cdots, s_r, t}{\bigoplus} H_*\rmod \left(Q^{m_1 \cdots m_r}_{s_1 \cdots s_r},  \widehat{H}_{t(d-1)}(C(\bullet, \rdbb), \qbb) \right).
\end{multline}

\begin{multline} \label{final_splitting_pi}
\qbb \otimes \pi_*(\emb) \cong \underset{s_1, \cdots, s_r, t}{\prod} H_*\rmod \left(Q^{m_1 \cdots m_r}_{s_1 \cdots s_r},  \qbb \otimes \widehat{\pi}_{t(d-2)+1}C(\bullet, \rdbb) \right)\\
\cong \underset{s_1, \cdots, s_r, t}{\bigoplus} H_*\rmod \left(Q^{m_1 \cdots m_r}_{s_1 \cdots s_r},  \qbb \otimes \widehat{\pi}_{t(d-2)+1}C(\bullet, \rdbb) \right).
\end{multline}

The product is replaced by the direct sum because only finitely many factors contribute for any given degree. We use here that $d > 2 \mathrm{max}\{m_i| \ 1 \leq i \leq r\} + 1$. This can be seen from the explicit description of our complexes analogously to the case $r=1$, see  Remark~\ref{rmq:finite_dim}. (For~\eqref{final_splitting_pi}, this is true even for a weaker constraint 
$d> \mathrm{max}\{m_i| \ 1 \leq i \leq r\} + 2$.) 
  The splitting by $s_1,\ldots,s_r$ will be called Hodge splitting and the splitting by $t$ will be called splitting by complexity. 

The current paper is the first one of our project. In the second one~\cite{songhaf_tur}, we determine the generating functions
of the Euler characteristics of the summands in~\eqref{final_splitting_H} and~\eqref{final_splitting_pi}. By the time
when our project was completed, another work by Fresse, Willwacher and the second author appeared~\cite{fr_tur_wilw},
where among other things, by very different techniques, the statement of Theorem~\ref{rational_homotopy_thm},
equivalently~\eqref{final_splitting_pi},  was improved for $r=1$ to codimension $d-m_1>2$.\footnote{More precisely they proved that the rational
homotopy groups of the little discs operads mapping space $\mathrm{hOper}({\mathcal B}_m,{\mathcal B}_d)$, $d-m>2$, are computed
by the (shifted) hairy graph-complex ${\mathcal E}_\pi^{m,d}$  that describes explicitly the right-hand side of~\eqref{final_splitting_pi}  (see Section~\ref{graph_complexes}). (In~\cite{fr_tur_wilw} this complex is denoted by $\mathrm{HGC}_{m,d}$.) The latter result implies ours
thanks to the delooping result $\overline{\mathrm{Emb}}_c(\rbb^m,\rbb^d)\simeq \Omega^{m+1} \mathrm{hOper}({\mathcal B}_m,{\mathcal B}_d)$ from~\cite{boav_weiss}, see also~\cite{duc_tur,dwyer_hess,dwyer_hess2,turchin14}.}  Using the classical study of the isotopy classes of higher dimensional spherical links due to A.~Haefliger and others~\cite{haefliger,cr_fer_skop},
we confirmed, see Theorem~\ref{t:low}, that~\eqref{final_splitting_pi}  holds at the level of $\pi_0$  always when codimension $>2$.  Based on these facts, we conjecture that this identity~\eqref{final_splitting_pi} is true for $*\geq 0$ 
and $\codimlow$, see Conjecture~\ref{con:low}.

To conclude the introduction we mention that the approach that we use is very powerful in the sense that it allows one to determine
all the rational homology and homotopy groups for links spaces in the stable range.  
 However,
the weak side of this approach, it is completely abstracted from the geometry of embedding spaces.  It is an interesting problem 
to find a geometrical interpretation of the graph-complexes produced by this approach. Remarkably, similar graph-complexes
appear in the perturbative approach (also called Bott-Taubes integration). This approach is well established for one-dimensional knots and links~\cite{CCL,volic}, in the sense that it produces a well defined map from a graph-complex computing the cohomology 
of the link space to its de Rham complex. However, it is still an open  conjecture whether this map is a 
quasi-isomorphim.\footnote{An attempt to construct geometrical cycles that pair non-trivially with the configuration integral graph-cocycles has been done in~\cite{longoni,pelatt_sinha}.} In the higher-dimensional case, this approach has  only partial results and only  for  knots (one component links): one was only able to produce a well-defined map for the loop order $\leq 1$ part
of a complex quasi-isomorphic to ours~\cite{sakai10,sakai_watanabe}.  It is also quite interesting that in those works, one only requires the codimension
to be $> 2$, which is another sign in support of our  Conjecture~\ref{con:low}.

\vspace{5mm}

\textbf{Outline of the paper.}

\begin{enumerate}

\item[-] In Section~\ref{rational_homology_string_link_section} we prove Theorem~\ref{rational_homology_thm}, which determines the rational homology of the space of high dimensional string links. The proof is based on a version, developed by Arone and the second author in \cite{aro_tur12}, of the Goodwillie-Weiss embedding calculus \cite{good_weiss99, weiss99,weiss04}. It uses the relative formality theorem (of the inlcusion ${\mathcal B}_n \hookrightarrow {\mathcal B}_d$ of the little $n$-disks operad inside the little $d$-disks operad) proved for $d\geq 2n+1$ in \cite{lam_vol} and improved to the range $d\geq n+2$ in~\cite{tur_wilw}. It also uses the fact that the right $\Omega$-module $\widetilde{C}_*((\vee_{i=1}^r S^{m_i})^{\wedge \bullet}, \qbb)$ is formal in the category $\underset{\Omega}{\mbox{Rmod}}$. 

\item[-]  In Section~\ref{graph_complexes} we prove Theorem~\ref{rational_homotopy_thm}, which determines the rational homotopy of $\emb$. The plan of the proof is the same as that of \cite[Theorem 1.1 (ii)]{aro_tur13}, for which our theorem is a generalization. 
In order to prove \cite[Theorem 1.1 (ii)]{aro_tur13}, the authors, Arone and the second author, explicitly describe \textit{hairy graph-complexes} computing the rational homology and the rational homotopy of the space $\overline{\mbox{Emb}}_c(\rbb^m, \rbb^d)$. Our complexes, described in Section~\ref{graphs_complexes_emh-empi}, are the same as theirs except that the external vertices of our graphs are colored by the components of the links.\footnote{Similar graph-complexes were studied in~\cite{CKV}.} They are obtained by taking the injective resolution of the
target $\Omega$-modules in~\eqref{final_splitting_H} and~\eqref{final_splitting_pi}. Theorem~\ref{rational_homotopy_thm}
follows from the fact that the homotopy hairy graph-complex  turns out to be exactly the subcomplex of connected
graphs of the homology hairy graph-complex. At the end of this section we explain how our graph-complexes are related
to the modular envelop of the ${\mathcal L}_\infty$ operad.

  \item[-]  Section~\ref{s:low} is devoted to  Conjecture~\ref{con:low} saying that our
hairy graph-complex computes the rational homotopy of $\emb$ for any codimension $> 2$. We show Theorem~\ref{t:low} that the statement of the
conjecture is true at the level of $\pi_0$. From the point of view of the hairy graph-complex only its tree part can contribute
to degree zero. Another way to say it  is that the rational additive invariants of higher dimensional 
string links can all be expressed as  generalized Milnor-Koschorke type invariants~\cite{casson,CohKS,habeg_masb00,kosch,munson11}.\footnote{We also believe that  such invariants can be 
 expressed
as Bott-Taubes type integrals similarly to the Haefliger invariant~\cite{sakai10}.}

\item[-] In Section~\ref{generating_function_homology} we  explicitly describe  graph-complexes computing the rational 
(co)homology and homotopy of $\emb$, which are obtained by taking projective resolution of the source $\Omega$-modules
in~\eqref{final_splitting_H} and~\eqref{final_splitting_pi}. We use here the fact that a right $\Omega$-module is the same thing as a right module over the
commutative non-unital operad ${{\mathcal C}om}_+$ and also the Koszul duality between ${{\mathcal C}om}_+$ and the operad 
${{\mathcal L}ie}$ of Lie algebras. The obtained complexes are thus called \textit{Koszul complexes}. We use them   in the follow up paper~\cite{songhaf_tur} for the computations of the Euler characteristics of the summands in~\eqref{final_splitting_H}
and~\eqref{final_splitting_pi}. We also show that natural truncations of the Koszul complexes can be used to determine
the stages  of the Goodwillies-Weiss towers $T_\bullet C_* (\emb,\qbb)$ and (conjecturally) $\qbb\otimes \pi_*T_\bullet\emb$.


\end{enumerate}

\textbf{Acknowledgements} 
This work has been supported by Fonds de la Recherche Scientifique-FNRS (F.R.S.-FNRS), that the authors acknowledge. It has been also supported by the Kansas State University (KSU), where this paper was partially written during the stay of the first author, and which he thanks for hospitality. 
The second author thanks S.~Melikhov, A.~Skopenkov, and M.~Skopenkov for communication.
The second author is currently supported by the Simons collaboration grant.

\section{The rational homology of $\emb$} \label{rational_homology_string_link_section}

The goal of this section is to prove Theorem~\ref{rational_homology_thm} announced in the introduction. 

The first and main step in proving Theorem~\ref{rational_homology_thm} is  Theorem~\ref{rational_homology_thm2}.
  Propositions~\ref{crucial_prop} 
and~\ref{p:equiv_towers}
below will be used in the proof of the latter theorem.  Define $m = \mbox{max} \{m_i| \ 1 \leq i \leq r\}$.  Without changing the homotopy type of $\embthm$ we can assume that the fixed affine embedding $\iota\colon \coprod_{i=1}^r \mathbb{R}^{m_i}\hookrightarrow\mathbb{R}^d$ can be factored through $\mathbb{R}^{m+1}$. Moreover let us  assume that 
$\iota$ restricted on each component $\mathbb{R}^{m_i}$ can be factored through an affine $m$-space $V_i
\subset \rbb^{m+1}$, so that all $V_i$, $i=1\ldots r$, are parallel to each other and disjoint. For a space $X$ let $\widehat{X}$ be its   one point compactification. Define $M_i$  as a tubular neighborhood of  $\widehat{\iota(\mathbb{R}^{m_i})}$ inside $\widehat{V}_i$ minus  the point at infinity. Simialrly define $M$ as an $(m+1)$-submanifold of $\mathbb{R}^{m+1}$ obtained as a regular neighborhood of the wedge of spheres
$\reallywidehat{\iota\left(\coprod_{i=1}^r \mathbb{R}^{m_i}\right)}$ in $\widehat{\mathbb{R}^{m+1}}$ minus $\infty$.

\begin{rmq}\label{r:manifolds}
We don\rq{}t need it for any of our arguments, but one can show that each $M_i$ is diffeomorphic to $\mathbb{R}^m\setminus S^{m-m_i-1}$, where $S^{m-m_i-1}$ is an unknotted $(m-m_i-1)$-sphere in $\rbb^m$. (By $(-1)$-sphere as usual we understand the empty set.) One can also show that 
$M$ is diffeomorphic to $\rbb^{m+1} \backslash \vee_{i=1}^r S^{m-m_i}$, where $\vee_{i=1}^r S^{m-m_i}$ is an unknotted wedge of spheres in $\rbb^{m+1}$.
\end{rmq}


\begin{prop} \label{crucial_prop}
The natural restriction maps 
\begin{eqnarray} \label{crucial_equiv}
\overline{\mathrm{Emb}}_c(M, \rdbb) \stackrel{\simeq}{\longrightarrow} \overline{\mathrm{Emb}}_c(\coprod_{i=1}^r M_i, \rdbb)\stackrel{\simeq}{\longrightarrow}\embthm  
\end{eqnarray}
are  weak equivalences.
\end{prop}
\begin{proof}
It is easy to see that the embedding spaces $\mathrm{Emb}_c(M, \rdbb)$ and 
$\mathrm{Emb}_c(\coprod_{i=1}^r M_i, \rdbb)$ (respectively immersion spaces $\mathrm{Imm}_c(M, \rdbb)$ and 
$\mathrm{Imm}_c(\coprod_{i=1}^r M_i, \rdbb)$) are homotopy equivalent to the spaces 
$\mathrm{Emb}^{(m+1)\text{-fr}}_c(\coprod_{i=1}^r \mathbb{R}^{m_i}, \rdbb)$ and
$\mathrm{Emb}^{m\text{-fr}}_c(\coprod_{i=1}^r \mathbb{R}^{m_i}, \rdbb)$ (respectively
$\mathrm{Imm}^{(m+1)\text{-fr}}_c(\coprod_{i=1}^r \mathbb{R}^{m_i}, \rdbb)$ and $\mathrm{Imm}^{m\text{-fr}}_c(\coprod_{i=1}^r \mathbb{R}^{m_i}, \rdbb)$)
of partially framed long embeddings  $\coprod_{i=1}^r \mathbb{R}^{m_i}\hookrightarrow\mathbb{R}^d$ (respectively 
immersions $\coprod_{i=1}^r \mathbb{R}^{m_i}\looparrowright\mathbb{R}^d$). By a \lq\lq partially framing\rq\rq{}
 of a long link we mean that the $i$th component of a link has $(m+1-m_i)$ in case of $M$ and $(m-m_i)$ in case of $\coprod_{i=1}^r M_i$ everywhere linearly independent sections of the normal bundle, which must coincide with a fixed constant partial framing outside a compact set. When we pass to the homotopy fiber 
 $\overline{\mathrm{Emb}}{}^{\ell\text{-fr}}_c(\coprod_{i=1}^r \mathbb{R}^{m_i}, \rdbb)$, $\ell=m+1$ or $m$,
or in other words look at the spaces of partially framed embeddings modulo partially framed immersions, this partial framing 
\lq\lq{}cancels out\rq\rq{}. To be precise, the natural projection 
\begin{equation}\label{eq:fibers}
\overline{\mathrm{Emb}}{}^{\ell\text{-fr}}_c(\coprod_{i=1}^r \mathbb{R}^{m_i}, \rdbb)\to\emb
\end{equation}
is a weak equivalence. It is a Serre fibration with the fiber over any point homeomorphic to the product of $r$ spaces. The $i$th factor of this product
is homeomorphic to the space of maps
\[
\mathbb{R}^{m_i}\times [0,1]\to SO(d-m_i)/SO(d-\ell)
\]
which are constant maps to the basepoint everywhere except a compact subset of $\mathbb{R}^{m_i}\times [0,1)$. Each such factor
is a contractible space. To see it, consider the precomposition with the homotopy
\[
h_\tau\colon \mathbb{R}^{m_i}\times [0,1]\to \mathbb{R}^{m_i}\times [0,1],\,\, \tau\in[0,1],
\]
defined as $h_\tau(x,t)= (x,(1-\tau)t+\tau)$ which deformation retracts $\mathbb{R}^{m_i}\times [0,1]$ onto 
$\mathbb{R}^{m_i}\times \{1\}$.
\end{proof}


One of the reasons we need all the components of our links to be of the same dimension is that we want to apply the 
 theory  \cite{good_weiss99,weiss99,weiss04} of the Goodwillie-Weiss manifold functor calculus. If $F$ is a functor from a category of open subsets of a manifold (which serves as an input in that theory), the $k$th approximation of $F$ is denoted $T_kF$, $k \leq \infty$. The sequence $\{ T_kF\}_{0 \leq k \leq \infty}$ is usually called the \textit{Taylor tower} associated to $F$, which converges, under good assumptions, to $F$. In our situation we apply the functor calculus to the categories of open sets of  $\mathbb{R}^{m+1}$, 
 respectively of $\coprod_{i=1}^r \mathbb{R}^m$, whose complement 
  is compact.

\begin{prop}\label{p:equiv_towers}
The restriction map $\overline{\mathrm{Emb}}_c(M, \rdbb) \stackrel{\simeq}{\longrightarrow} \overline{\mathrm{Emb}}_c(\coprod_{i=1}^r M_i, \rdbb)$ induces an equivalence of Goodwillie-Weiss towers
$$
T_kC_*\overline{\mathrm{Emb}}_c(M, \rdbb) \stackrel{\simeq}{\longrightarrow} T_kC_*\overline{\mathrm{Emb}}_c(\coprod_{i=1}^r M_i, \rdbb), \quad  0\leq k\leq \infty.
$$
\end{prop}

\begin{proof}
By construction~\cite[Section~6]{weiss99} for  $k<\infty$, the functors $T_kC_*\overline{\mathrm{Emb}}_c(-, \rdbb)$ and  
$T_kC_*\overline{\mathrm{Emb}}_c(\coprod_{i=1}^r (-)_i, \rdbb)$ are the best degree $k$ polynomial approximations 
to the contravariant functors
\[
C_*\overline{\mathrm{Emb}}_c(-, \rdbb)\colon \widetilde{O}(\mathbb{R}^{m+1})\to \mathrm{Ch}_*
\]
 and  
 \[
C_*\overline{\mathrm{Emb}}_c(\coprod_{i=1}^r (-)_i, \rdbb)\colon \widetilde{O}(\coprod_r\mathbb{R}^m)\to \mathrm{Ch}_*,
\]
 respectively. Here $\mathrm{Ch}_*$ is the model category of chain complexes over rational numbers;  $\widetilde{O}(\mathbb{R}^{m+1})$ and $\widetilde{O}(\coprod_r\mathbb{R}^m)$ are the categories of open subsets of $\mathbb{R}^{m+1}$ 
 and $\coprod_r\mathbb{R}^m$, respectively, whose complement is compact.
 Since both these functors are polynomial of degree~$k$,  both complexes $T_kC_*\overline{\mathrm{Emb}}_c(M, \rdbb)$ and  
$T_kC_*\overline{\mathrm{Emb}}_c(\coprod_{i=1}^r M_i, \rdbb)$ are obtained as a homotopy limit over a finite diagram encoding the handle attachments  and iteration of the limits over corresponding $(k+1)$-subcubical diagrams, 
see~\cite[Definition~2.2]{weiss99}. One has a natural transformation induced by restriction between the two diagrams which is an objectwise equivalence. As a consequence the homotopy limits must be also equivalent.
  Because homotopy inverse limits of equivalences are equivalences, the $k=\infty$ case follows. 
\end{proof}
 

 Recalling that if $X$ is a pointed topological space, then the sequence $X^{\times\bullet} = \{ \underbrace{X \times \cdots \times X}_{k}\}_{k \geq 0}$ is equipped with a natural structure of a right $\Gamma$-module: to every finite pointed set we assign the space of base preserving maps from it to $X$.   We have the following theorem  in which the notations "$\underset{\Gamma}{\mbox{hRmod}}$" and "$C(\bullet, \rdbb)$" were  introduced in Section~\ref{s:intro}.

\begin{thm} \label{rational_homology_thm2}
For $d > 2 \mathrm{max}\{m_i| \ 1 \leq i \leq r\} + 1$, there is a quasi-isomorphism
\[C_*(\embthm) \otimes \mathbb{Q}  \simeq \underset{\Gamma}{\mathrm{hRmod}} \left(C_*((\vee_{i=1}^r S^{m_i})^{\times\bullet}, \qbb), H_*(C(\bullet, \rdbb), \qbb) \right).\]
\end{thm}

\begin{proof}
Let us start with the following weak equivalences.
\begin{eqnarray} \label{weak_equiv1}
\begin{array}{cccc}
C_*\emb & \simeq & C_*\overline{\mbox{Emb}}_c(\coprod_{i=1}^rM_i, \rdbb) &   \mbox{by Proposition}~\ref{crucial_prop} \\
 & \simeq &  T_\infty C_*\overline{\mbox{Emb}}_c(\coprod_{i=1}^rM_i, \rdbb) &\\
 & \simeq & T_\infty C_*\overline{\mbox{Emb}}_c(M, \rdbb)&  \mbox{by Proposition}~\ref{p:equiv_towers}.
\end{array}
\end{eqnarray}
The second equivalence is the fact that the Goodwillie-Weiss tower for singular chains of embeddings converges for $d>2m+1$ \cite{weiss04}.
 
 By the Arone-Turchin work \cite{aro_tur12}, one has \cite[Proposition 2.11]{songhaf13}, which states that if $N \subseteq \rbb^n$ is the complement of a compact subset of $\rbb^n$, then 
 \begin{eqnarray} \label{codim_condition}
  T_{\infty} C_*(\overline{\mbox{Emb}}_c(N, \rdbb), \qbb)  \simeq 
  \underset{\Gamma}{\mathrm{hRmod}} (C_*(\widehat{N}^{\times\bullet}), H_*(C(\bullet, \rdbb); \qbb) ), 
  \end{eqnarray}
  where $\widehat{N}$ is  is an open subset of the one-point compactification of $\rbb^n$ consisting of $N$ and $\infty$. This space is pointed at $\infty$. In~\cite[Proposition 2.11]{songhaf13} this result is stated in the range $d>2n+1$ as it uses the relative formality of the little discs operads which was proved at that time only in this range~\cite{lam_vol}. As it turned the relative formality holds in larger range  $d>n+1$ \cite{tur_wilw} implying~\eqref{codim_condition} for $d>n+1$.

By rewriting (\ref{codim_condition})  with our manifold $M \subseteq \rbb^{m+1}$, we obtain the following weak equivalence (for  $d > m +2$): 
 \begin{eqnarray} \label{codim_condition2}
  T_{\infty} C_*(\overline{\mbox{Emb}}_c(M, \rdbb), \qbb)  \simeq 
  \underset{\Gamma}{\mathrm{hRmod}} (C_*(\widehat{M}^{\times\bullet}), H_*(C(\bullet, \rdbb); \qbb) ). 
  \end{eqnarray}
Combining now (\ref{weak_equiv1}) and (\ref{codim_condition2}), and using the fact that $\widehat{M}$ is homotopy equivalent to $\vee_{i=1}^r S^{m_i}$ we get exactly the statement of our theorem.
\end{proof}

Now we pass from $\Gamma$-modules to $\Omega$-modules. In short we do it as it reduces computatons.
In \cite{pira00} Pirashvili constructs a functor (called the \textit{cross effect}) 
\begin{eqnarray} \label{cross_effect_functor}
cr \colon \underset{\Gamma}{\mbox{Rmod}}  \lra \underset{\Omega}{\mbox{Rmod}}
\end{eqnarray}
 that turns out to be an equivalence of categories.  This functor is defined as follows. Let $L$ be a right $\Gamma$-module. 
 Let $k_+$ be the set $\{0,1,\ldots,k\}$ pointed in~$0$, and $k$ denote the set $\{1,\ldots,k\}$. Let 
 $
 \alpha_i\colon k_+\to (k-1)_+
 $
 be the map 
 \[
 \alpha_i(j)=
 \begin{cases}
 j,&j<i;\\
 0,&j=i;\\
 j-1,&j>i.
 \end{cases}
 \]
 Then the component $cr L(k)$ is the quotient of $L(k_+)$ by the sum of the images of the induced maps $$\alpha_i^*\colon
  L((k-1)_+)\to L(k_+).$$

    Define  
\[ \widehat{H}_*(C(\bullet, \rdbb), \qbb) = cr \left( H_*(C(\bullet, \rdbb), \qbb) \right).\] 
By noticing that the cross effect of $C_*((\vee_{i=1}^r S^{m_i})^{\times\bullet}, \qbb)$ is equivalent to $\widetilde{C}_*((\vee_{i=1}^r S^{m_i})^{\wedge \bullet}, \qbb)$, where $\widetilde{C}_*(-)$ denotes the reduced singular chain functor, and "$\wedge$" is the usual smash product of spaces, we have the following corollary. 

\begin{coro} \label{rational_homology_coro}
For $d > 2 \mathrm{max}\{m_i| \ 1 \leq i \leq r\} + 1$, there is a quasi-isomorphism
\[C_*(\embthm) \otimes \mathbb{Q}  \simeq \underset{\Omega}{\mathrm{hRmod}} \left(\widetilde{C}_*((\vee_{i=1}^r S^{m_i})^{\wedge \bullet}, \qbb), \widehat{H}_*(C(\bullet, \rdbb), \qbb) \right).\]
\end{coro}

To finish the proof of Theorem~\ref{rational_homology_thm} we only need to show that the $\Omega$-module 
$\widetilde{C}_*((\vee_{i=1}^r S^{m_i})^{\wedge \bullet})$ is rationally formal. This follows from  
Lemma~\ref{formality_lem} and also the fact that each  $\widetilde{C}_*((S^{m_i})^{\wedge \bullet})$, $1 \leq i \leq r$, is formal (by \cite[Lemma 8.5]{aro_tur12}).

\begin{lem} \label{formality_lem} Let $X$ and $Y$ be two pointed topological spaces. Assume that  the right $\Omega$-modules $\widetilde{C}_*(X^{\wedge \bullet})$ and $\widetilde{C}_*(Y^{\wedge \bullet})$ are formal in the category $\underset{\Omega}{\mathrm{Rmod}}$ of right $\Omega$-modules. Then $\widetilde{C}_*((X \vee Y)^{\wedge \bullet})$ is formal in the same category.
\end{lem}

\begin{proof}
Notice first that for each $k \geq 0$ we have 
$$(X \vee Y)^{\wedge k} \simeq \underset{f \colon \{1, \cdots, k \} \lra \{1,2 \}}{\bigvee} X^{\wedge f^{-1}(1)} \wedge Y^{\wedge f^{-1}(2)}.$$ 
This implies 
\begin{multline}\label{eq:wedge}
\widetilde{C}_*((X \vee Y)^{\wedge k} )  \simeq
 \underset{f \colon \{1, \cdots, k \} \lra \{1,2 \}}{\bigoplus}
 \widetilde{C}_*(X^{\wedge f^{-1}(1)}\wedge Y^{\wedge f^{-1}(2)})
 \simeq\\
   \underset{f \colon \{1, \cdots, k \} \lra \{1,2 \}}{\bigoplus}
 \widetilde{C}_*(X^{\wedge f^{-1}(1)}) \otimes \widetilde{C}_*(Y^{\wedge f^{-1}(2)}).
\end{multline}

Consider the contravariant functor $\lambda \colon \Omega \times \Omega \lra \mbox{Ch}_*$, from $\Omega \times \Omega$ to chain complexes, defined by $\lambda(p,q) = \widetilde{C}_*(X^{\wedge p}) \otimes \widetilde{C}_*(Y^{\wedge q}).$
Consider also the  functor $\Delta \colon \underset{\Omega \times \Omega}{\mathrm{Rmod}}\to \underset{\Omega}{\mathrm{Rmod}}$ defined by 
\[\Delta (\mu)(k) = \underset{f \colon \{1, \cdots, k \} \lra \{1,2 \}}{\bigoplus} \mu(f^{-1}(1), f^{-1}(2)), \]
where $\mu$ is a right $\Omega \times \Omega$ module.  By~\eqref{eq:wedge}, $\Delta(\lambda)$ is equivalent to $\widetilde{C}_*((X \vee Y)^{\wedge \bullet})$ whose formality we want to show.
Since $\lambda$ is formal by hypothesis, that is there exists a  zigzag of natural quasi-isomorphisms, $\lambda \simeq H_*(\lambda)$, between $\lambda$ and $H_*(\lambda)$, it follows that 
$$
\Delta (\lambda) \simeq \Delta (H_*(\lambda)) = H_*(\Delta (\lambda)).
$$
\end{proof}


%
%
%
%
%
%
%
%
%
%
%
%
%

\section{Hairy graph-complexes} \label{graph_complexes}

The goal of this section is to prove Theorem~\ref{rational_homotopy_thm} announced in the introduction.  At the end we also explain how our graph-complexes are related to the modular envelop of ${\mathcal L}_\infty$.


\subsection{Graph-complexes $\empi$ and $\emh$} \label{graphs_complexes_emh-empi}

In this subsection we will explicitly describe  two graph-complexes: $\emh$ and $\empi$. In the next subsection we will show that they respectively compute the rational homology and the rational homotopy of $\emb$. Our description is exactly the same as that of complexes $\mathcal{E}^{m,n}_H$ and $\mathcal{E}_{\pi}^{m, n}$ introduced in \cite[Section 2.1]{aro_tur13}, except that our graphs have their external vertices "decorated" or "colored" with $\{1, \cdots, r\}$ as the set of colours. As before the colours correspond to the components of the links.


Let us start with $\empi$, which is a differential graded vector space  spanned by some connected {\it hairy colored graphs} defined below, modulo some relations.  A \textit{hairy graph} $G$ is a triple $(V_{G}^E,V_{G}^I, E_{G})$ where 
\begin{enumerate}
\item[-] $V_{G}^E$ is a finite non-empty set whose elements are non-labeled {\it external} vertices of valence $1$; 
\item[-] $V_{G}^I$ is a finite set (which can be empty) whose elements are non-labeled {\it internal} vertices of valence $\geq 3$;
\item[-]  $E_{G}$ is the set of oriented edges, which may include tadpoles and multiple edges. 
\end{enumerate}

We require that the graphs are connected. A \textit{hairy colored graph} is a couple $(G, f_{G})$ where $f_{G} \colon V_{G}^E \lra \{1, \cdots, r\}$ is a map from the set of external vertices to the set of colours. 
 Given a colored graph $G$, an \textit{orientation set} is  the union  $ V_{G}^E \cup V_{G}^I \cup E_{G}$. 
  Elements of $V_G^I$ are assigned degree $-d$; elements of $V_{G}^E$ of color $i$ are assigned degree $-m_i$; elements of $E_G$ are assigned degree $d-1$.  We will say that a colored graph is \textit{oriented} if it is equipped with a total order $O_G$ of its orientation set. Two graphs $(G, f_G, O_G)$ and $(G', f_{G'}, O_{G'})$ are equivalent if there are two bijections 
$$ g \colon V_G^E \cup V_G^I=V_G \stackrel{\cong}{\lra} V_{G'} = V_{G'}^E \cup V_{G'}^I \qquad \mbox{and} \qquad h \colon E_G \stackrel{\cong}{\lra} E_{G'},$$ 
respecting adjacency, orientation of edges,  order of orientation sets, and coloration. (Since $g$ preserves the valence,  it must send external vertices to external and internal to internal.)
 Let $\mg$ be the collection of all equivalence classes of oriented colored graphs, and let $\empi$ denote the vector space  spanned by $\mg$ modulo  two relations in $\qbb [\mg]$: 

\begin{enumerate}
\item[-] $(G, f_G, O_G) = (-1)^d(G', f_{G'}, O_{G'})$ if $(G, f_G, O_G)$ differs from $(G', f_{G'}, O_{G'})$ only by an orientation of an edge; 
\item[-] $(G, f_G, O_G) = \pm (G', f_{G'}, O_{G'})$ if $(G', f_{G'}, O_{G'})$ is obtained from $(G, f_G, O_G)$ only by a permutation of the orientation set. The sign "$\pm$" is the Koszul sign of permutation taking into account the degrees of  elements.  
\end{enumerate}
 Given a colored hairy graph $(G,f_G)$, we define its degree as a sum of degrees of the elements in its orientation set:
\[ \mbox{degree}(G) = (d-1) |E_G|-d |V_G^I| - \sum_{i=1}^r m_i |f_G^{-1}(i)|.\]
This endows $\empi$ with a grading. It is also equipped with a differential $\partial$ that turns it into a chain complex. The differential of a graph is defined as the sum of expansion of its internal vertices.  More precisely, let $(G, f_G, O_G)$
be an oriented colored graph. For any internal vertex  $x \in V_G^I$ of valence $l$, let $(G_{px}, f_{G_{px}}, O_{G_{px}}), 1 \leq p \leq n_x$, denote the  graphs obtained from the expansion of $x$. One can see that the number of such graphs is 
$n_x = 2^{l-1}-l-1$. They are defined as follows. Notice first that an expansion produces a new vertex and a new edge. For $1 \leq p \leq n_x$, one has
\[G_{px} = (V^E_{G_{px}}, V^I_{G_{px}}, E_{G_{px}}) = (V^E_G, V^I_G \cup \{v_{px}\}, E_G \cup \{e_{px}\}), \quad f_{G_{px}} = f_G,\]
and the orientation set has two more elements: $v_{px}$ -- the new vertex, and $e_{px}$ -- the new edge (oriented from the old vertex to the new one). Its order is defined as follows: $v_{px} \leq e_{px}$ and for all $y$ in the orientation set of $G$, one has $e_{px} \leq y$ (here "$\leq$" denotes the total order). 

 The differential is then defined by 
\[\partial (G, f_G, O_G) = \sum_{x \in V_G^I} \sum_{p=1}^{n_x} (G_{px}, f_{G_{px}}, O_{G_{px}}).\]

\begin{figure}[!ht]
\centering
\includegraphics[scale=0.6]{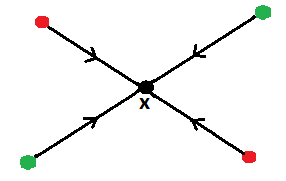}
\caption{} \label{graph_diff}
\end{figure}

Note that there is no sign in each term of the differential. It is hidden in the way we order the orientation set and orient the new edge. Let us see one example. Assume that $r=2$, that is, we have $2$ colours, and consider  the graph in Figure~\ref{graph_diff}. The three graphs that appear in the differential of Figure~\ref{graph_diff}  are Figure~\ref{graph1_diff}, Figure~\ref{graph2_diff} and Figure~\ref{graph3_diff}.

\begin{figure}[!ht]
\centering
\begin{minipage}[t]{3cm}
\includegraphics[scale=0.6]{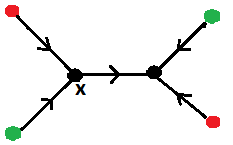}
\caption{} \label{graph1_diff}
\end{minipage}
\begin{minipage}[t]{1.5cm}
\includegraphics[scale=0.1]{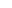}
\end{minipage}
\begin{minipage}[t]{3cm}
\includegraphics[scale=0.6]{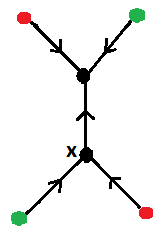}
\caption{} \label{graph2_diff}
\end{minipage}
\begin{minipage}[t]{1.5cm}
\includegraphics[scale=0.1]{blanc}
\end{minipage}
\begin{minipage}[t]{3cm}
\includegraphics[scale=0.6]{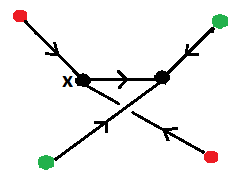}
\caption{} \label{graph3_diff}
\end{minipage}

\end{figure}

To see that for any graph $G\in \empi$, one has $\partial^2 G=0$, we notice that $\partial^2G$ is the sum of graphs obtained from
$G$ by inserting two new edges (if we contract these two edges, one should get $G$). Each such summand appears in $\partial^2G$ 
twice as it matters the order in which the two new edges are created. From the definition of $\partial$  these two repeated summands appear with opposite signs, because in the orientation set
\[
(v_{p_1x_1}e_{p_1x_1})(v_{p_2x_2}e_{p_2x_2})= - (v_{p_2x_2}e_{p_2x_2})(v_{p_1x_1}e_{p_1x_1})
\]
as the degree of $(v_{p_ix_i}e_{p_ix_i})$, $i=1,2$, is $-d+(d-1)=-1$.

Now we define $\emh$ as the cofree cocommutative coalgebra cogenerated by $\empi$. We can view $\emh$ as a vector space spanned by all possibly empty or disconnected colored hairy graphs (with each connected component from $\empi$).

\subsection{The rational homotopy of $\emb$} \label{proof_rational_homotopy_thm-section}

The goal of this subsection is to prove Theorem~\ref{rational_homotopy_thm}, which determines the rational homotopy of the space $\emb$.  The plan of the proof is as follows.  We will first prove 
 that the graph-complex $\emh$ is quasi-isomorphic to the complex $\underset{\Omega}{\mbox{hRmod}}(\widetilde{H}_*((\vee_{i=1}^r S^{m_i})^{\wedge \bullet}, \qbb), \widehat{H}_*({\mathcal B}_d(\bullet), \qbb))$ (where 
  ${\mathcal B}_d$ denotes the little $d$-disks operad). The latter complex is the derived mapping space. Since both the source and target have trivial differential, the complex itself is equivalent to a product of Ext groups that can be either computed by taking the projective resolution of the source 
(as in Section~\ref{generating_function_homology})
or by taking the injective resolution of the target.  We show that the second approach produces exactly the complex $\emh$.      
  A similar argument shows that $\empi$ is equivalent to the derived mapping space $\underset{\Omega}{\mbox{hRmod}}(\widetilde{H}_*((\vee_{i=1}^r S^{m_i})^{\wedge \bullet}, \qbb), \qbb\otimes\widehat{\pi}_*{\mathcal B}_d(\bullet))$. On the other hand,
 since $\emh$ computes $H_*(\emb,\qbb)$, by dimensional reasons we conclude that $\empi$ computes $\qbb\otimes \pi_*
 \emb$. Which finishes the proof of Theorem~\ref{rational_homotopy_thm}.

In \cite{aro_tur13} Arone and the second author explicitly describe a differential graded right  $\Omega$-modules $\dhat^{\bullet}$ that turns out to be an injective resolution of  $\widehat{H}_*({\mathcal B}_d(\bullet), \qbb)$.  We will first recall the right $\Gamma$-module $D_d^{\bullet}$, generated by  graphs, introduced in \cite{turchin10}. For $k \geq 0$, $D_d^k$ is a vector space  spanned by possibly disconnected graphs with $k$ {\it external} vertices labeled by $\{1,\ldots,k\}$ of any valence $\geq 0$ and some possibly empty set of non-labeled internal vertices of valence $\geq 3$, such that  each their connected component  has at least one external vertex.
 Tadpoles and multiple edges are allowed. The graphs are oriented, and an orientation set of such graph consists of the union of edges (considered as elements of degree $d-1$) and internal vertices (considered as elements of degree $-d$). Orienting a graph means putting a total order on its orientation set. As in Section~\ref{graphs_complexes_emh-empi}, there is a similar notion of equivalence of graphs, and similar sign relations. Again as in Section~\ref{graphs_complexes_emh-empi}, the differential is the sum of expansions of vertices. The expansion is applied to both external and internal vertices. The expansion of an internal vertex is as in Section~\ref{graphs_complexes_emh-empi}, while the expansion of an external vertex produces an external vertex (with the same label) and an internal one. 

It is well known that the sequence $D^{\bullet}_d = \{D_d^k\}_{k \geq 0}$ is an operad (Figure~\ref{insertion_operation_figure} shows an example of the composition) in differential graded cocommutative coalgebras (Figure~\ref{coproduct_figure} is an example of the coproduct in $D_d^3$).

\begin{figure}[!ht]
\centering
\includegraphics[scale=0.4]{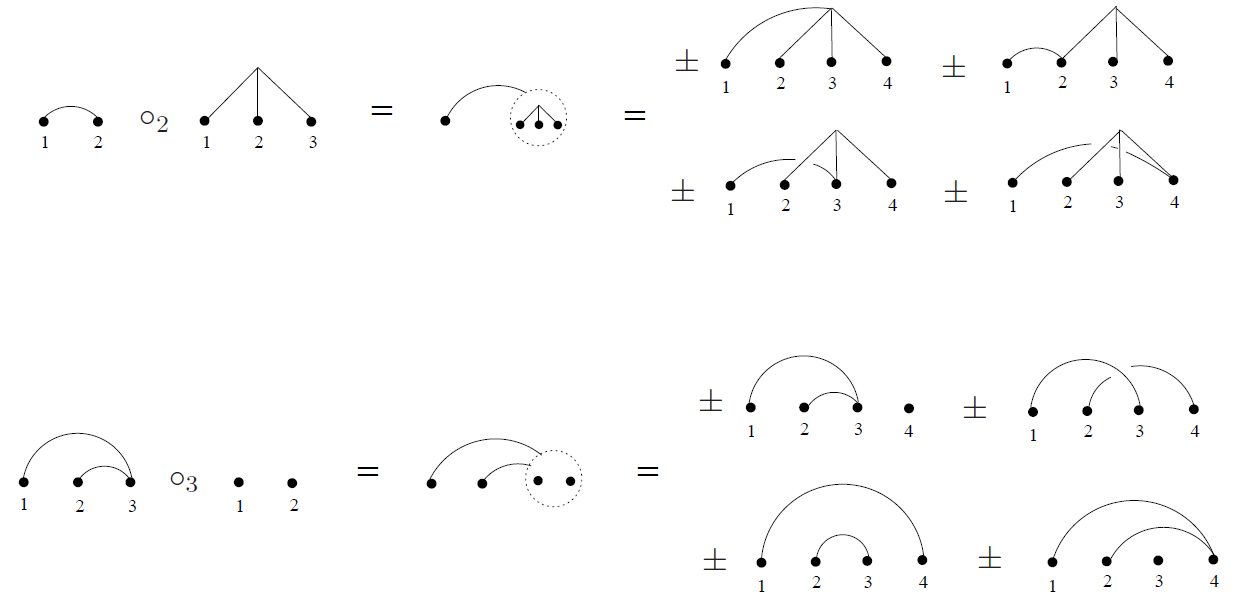}
\caption{Examples of the insertion operation} \label{insertion_operation_figure}
\end{figure}

\begin{figure}[!ht]
\centering
\includegraphics[scale=0.4]{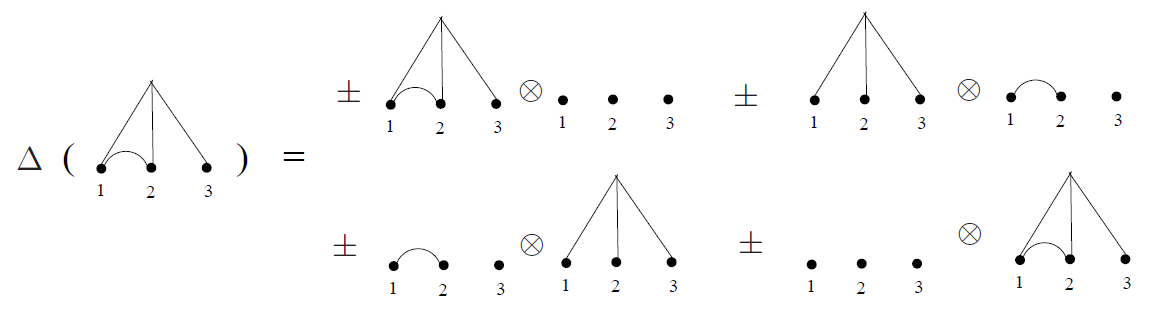}
\caption{An example of the coproduct in $D^3_d$} \label{coproduct_figure}
\end{figure}

It is also well known \cite{lam_vol, turchin10} that there is an inclusion of operads (in coalgebras) $H_*({\mathcal B}_d(\bullet), \qbb) \hookrightarrow D_d^{\bullet}$ that turns out to be a quasi-isomorphism. This inclusion implies that $D_d^{\bullet}$ is an infinitesimal bimodule over $H_*({\mathcal B}_d(\bullet), \qbb)$, and therefore over $H_0({\mathcal B}_d(\bullet), \qbb) = {{\mathcal C}om}$. By \cite[Proposition 4.9]{aro_tur12}, we deduce that $D_d^{\bullet}$ is a right $\Gamma$-module. Moreover, since the right $\Gamma$-module structure respects the coalgebra structure, $D^{\bullet}_d$ is a right $\Gamma$-module in coalgebras. Recalling the cross effect functor $cr$ from (\ref{cross_effect_functor}), one defines the right $\Omega$-module $\dhat^{\bullet}$ by 
\[ \dhat^{\bullet} = cr (D_d^{\bullet}).\]
It is easy to see (by the definition of the cross effect) that for each $k \geq 0$, $\dhat^k$ is spanned by the same generators as $D_d^k$, except that generators of $\dhat^k$ are graphs whose external vertices are of valence $\geq 1$. 
The following result is~\cite[Proposition 2.6]{aro_tur13}.

\begin{prop}\label{injective_resolution_prop} 
The right $\Omega$-module $\dhat^{\bullet}$ is finite dimensional and injective in each homological degree. 
The inclusion $\widehat{H}_*({\mathcal B}_d(\bullet), \qbb) \hookrightarrow \dhat^{\bullet}$ is a quasi-isomorphism of $\Omega$-modules
and thus is an injective resolution for every given complexity.
\end{prop}

\begin{rmq}\label{rm:complexity_homol}
To recall the decomposition by complexity $t$ in $\widehat{H}_*({\mathcal B}_d(\bullet), \qbb)$ is 
$
\bigoplus_t \widehat{H}_{t(d-1)}({\mathcal B}_d(\bullet), \qbb).
$
In $\dhat^\bullet$ one defines a similar splitting saying that a complexity of a graph $G\in\dhat^k$ is the first Betti number
of the graph obtained from $G$ by gluing together all external vertices.
\end{rmq}

\begin{thm} \label{emh_hrmod_thm}
For $\codim$ the graph-complex $\emh$ is quasi-isomorphic to $\rmodthm(\widetilde{H}_*((\vee_{i=1}^r S^{m_i})^{\wedge \bullet}, \qbb), \widehat{H}_*({\mathcal B}_d(\bullet), \qbb))$, that is, 
\begin{eqnarray} \label{emh_hrmod_formula}
 \emh \simeq \rmodthm(\widetilde{H}_*((\vee_{i=1}^r S^{m_i})^{\wedge \bullet}, \qbb), \widehat{H}_*({\mathcal B}_d(\bullet), \qbb)).
\end{eqnarray}
\end{thm}

\begin{proof}
The steps of the proof are the same as those of the proof of Theorem 2.2 from \cite{aro_tur13}. Thus, for more details, we refer the reader to \cite{aro_tur13}. For each $k \geq 0$, let $M(D_d^k)$ denote the subspace of $D_d^k$ generated by graphs whose external vertices are of valence $=1$. This space is endowed with a differential being viewed  as a quotient-complex of $D_d^k$.  We have the following quasi-isomorphisms
\begin{align*}
& \rmod(\widetilde{H}_*((\vee_{i=1}^r S^{m_i})^{\wedge \bullet}, \qbb), \widehat{H}_*({\mathcal B}_d(\bullet), \qbb))   \\ 
 & \simeq  \ \underset{\Omega}{\mbox{Rmod}}(\widetilde{H}_*((\vee_{i=1}^r S^{m_i})^{\wedge \bullet}, \qbb), \dhat^{\bullet}) \ \mbox{by} \ \mbox{Proposition}~\ref{injective_resolution_prop} \\
 & \simeq \bigoplus_{k=0}^{\infty} \mbox{hom}_{\Sigma_k} \left(\widetilde{H}_*((\vee_{i=1}^r S^{m_i})^{\wedge k}, \qbb),   M(D_d^k)\right)   \ \mbox{by} \ \cite[\mbox{Lemma 2.7 and Lemma 2.8}]{aro_tur13}  \\
& \simeq  \bigoplus_{k=0}^{\infty} \mbox{hom}_{\Sigma_k} \left( \bigoplus_{s_1+\cdots+s_r=k} \mbox{Ind}^{\Sigma_k}_{\Sigma_{s_1} \times \cdots \times \Sigma_{s_r}}\widetilde{H}_* ( S^{s_1m_1+\cdots + s_rm_r},\qbb), M(D_d^k)\right) \ \mbox{by} \ (\ref{homology_wedge}) \\
& \simeq \bigoplus_{k=0}^{\infty} \bigoplus_{s_1+\cdots +s_r =k} \mbox{hom}_{\Sigma_{s_1}\times\ldots\times\Sigma_{s_r}} \left(\widetilde{H}_* ( S^{s_1m_1+\cdots + s_rm_r}, \qbb), M(D_d^k) \right). 
\end{align*}

By the description of the graph-complex $\emh$ from Section~\ref{graphs_complexes_emh-empi}, it is not difficult to see that there is an isomorphism of chain complexes
\begin{equation}\label{eq:emh_M(D)}
 \emh \cong\bigoplus_{k=0}^{\infty} \bigoplus_{s_1+\cdots +s_r =k} \mbox{hom}_{\Sigma_{s_1}\times\ldots\times\Sigma_{s_r}} \left(\widetilde{H}_* ( S^{s_1m_1+\cdots + s_rm_r}, \qbb), M(D_d^k) \right),
 \end{equation}
 which completes the proof. 
 \end{proof}
 
 One should mention that the decomposition over $(s_1,\ldots,s_r)$ above appeared exactly as the multidegree Hodge splitting. 
 
 \begin{rmq}\label{rmq:finite_dim}
 A careful reader might prefer to see a product instead of the direct sum used in the above equations. But it is easy to show that $\emh$ is finite dimensional in every homological degree, for which reason it would not matter whether we take a product or a direct sum of its components, see \cite[Remark 13.3]{aro_tur12}  and  \cite[Section 5]{aro_tur13} for degree estimates in  similar graph-complexes. 
\end{rmq}

The following corollary is an immediate consequence of Theorem~\ref{rational_homology_thm} and Theorem~\ref{emh_hrmod_thm}. 

\begin{coro} \label{emh_emb_coro}
For $\codim$ the graph-complex $\emh$ computes the rational homology of the space $\embthm$, that is, 
\begin{eqnarray} \label{emh_emb_formula}
H(\emh) \cong H_*(\embthm, \qbb). 
\end{eqnarray}
\end{coro}

\begin{coro} \label{empi_emb_coro}
For $\codim$ the graph-complex $\empi$ computes the rational homotopy of the space $\embthm$, that is, 
\begin{eqnarray} \label{empi_emb_formula}
H(\empi) \cong \qbb \otimes \pi_*(\embthm). 
\end{eqnarray}
\end{coro}

\begin{proof}
The desired result is a combination of Corollary~\ref{emh_emb_coro} and the following two facts:
\begin{enumerate}
\item[-] the rational homology of the space $\emb$ is a graded polynomial coalgebra cogenerated by the rational homotopy of the same space;
\item[-] the subcomplex of $\emh$, formed by primitive elements, is exactly the complex $\empi$. 
\end{enumerate}
These two facts imply that (\ref{emh_emb_formula}) and (\ref{empi_emb_formula}) are equivalent. 
\end{proof}

\begin{thm} \label{empi_hrmod_thm}
For $\codim$ the graph-complex $\empi$ is quasi-isomorphic to $\rmodthm(\widetilde{H}_*((\vee_{i=1}^r S^{m_i})^{\wedge \bullet}, \qbb), \qbb \otimes \widehat{\pi}_*{\mathcal B}_d(\bullet))$.
\end{thm}

\begin{proof}
The proof goes in the same way as that of Theorem~\ref{emh_hrmod_thm}. Instead of $\dhat^{\bullet}$, one constructs a complex of injective $\Omega$-modules $\phat^{\bullet} \simeq \qbb \otimes \widehat{\pi}_*{\mathcal B}_d(\bullet)$ (see \cite[Proposition 2.4]{aro_tur13} or \cite{turchin10}) defined as the cross effect of the right $\Gamma$-module $P_d^{\bullet} = \{P_d^k\}_{k \geq 0}$. Here $P_d^k$ is the primitive part of $D_d^k$ -- it is spanned by so called {\it internally connected graphs}, that have exactly one connected component different from singleton.  It is also important that the quasi-isomorphism 
$\phat^{\bullet} \simeq \qbb \otimes \widehat{\pi}_*{\mathcal B}_d(\bullet)$ respects the splitting by complexity $t$. In $\phat^\bullet
\subset\dhat^\bullet$ it is recalled in Remark~\ref{rm:complexity_homol}, while for the right-hand side it is
$\qbb \otimes \widehat{\pi}_*{\mathcal B}_d(\bullet)\simeq\oplus_t\qbb \otimes \widehat{\pi}_{t(d-2)+1}{\mathcal B}_d(\bullet)$.
\end{proof}

%

We finally notice that Theorem~\ref{rational_homotopy_thm} is an  immediate consequence of Corollary~\ref{empi_emb_coro} and Theorem~\ref{empi_hrmod_thm}.

\subsection{Homotopy graph-complex $\empi$ and the modular envelop of ${\mathcal L}_\infty$}\label{ss:MOD(L)}
Similarly to \eqref{eq:emh_M(D)}, the complex $\empi$ splits into a direct sum of complexes
\begin{equation}\label{eq:empi_M(P)}
 \empi \cong\bigoplus_{k=1}^{\infty} \bigoplus_{s_1+\cdots +s_r =k} \mbox{hom}_{\Sigma_{s_1}\times\ldots\times\Sigma_{s_r}} \left(\widetilde{H}_* ( S^{s_1m_1+\cdots + s_rm_r}, \qbb), M(P_d^k) \right),
 \end{equation}
where $P_d^k$ is defined in the proof of Theorem~\ref{empi_hrmod_thm}. Similarly to $M(D_d^k)$, one defines 
$M(P_d^k)$ as the quotient complex of $P_d^k$ spanned by graphs in $P_d^k$ whose all external vertices are univalent.

The symmetric sequence $M(P_d^\bullet)$ is a sequence of graph-complexes that up to a regrading are components of the 
modular envelop of ${\mathcal L}_\infty$. We briefly recall this connection, while we provide more details in the 
follow up paper~\cite{songhaf_tur}. 

Cyclic and modular operads were introduced by E.~Getzler and M.~Kapranov~\cite{getzler_kapranov95,getzler_kapranov98}. 
We will only consider the case when the underlying category is that of chain complexes.
In short, a {\it cyclic operad} is a usual one $O=\{O(n),\, n\geq 0\}$, for which the output in its elements has the same role as the
inputs. In particular $O(n)$ is acted on by $\Sigma_{n+1}$. To distinguish the cyclic arity with the usual one, one writes $O((n))$
for $O(n-1)$. A {\it modular operad} is a collection of objects $M=\{M((g,n)),\, g\geq 0, \, n\geq 0\}$, such that
$\{M((n)) =\oplus_gM((g,n)),\, n\geq 0\}$ is a cyclic operad (allowing now cyclic arity 0), and in addition to this structure one
also has self-composition maps $M((g,n))\to M((g+1,n-2))$, satisfying natural axioms.\footnote{For more details,
see~\cite{getzler_kapranov98,hin_vaintrob02,markl01}. Getzler-Kapranov~\cite{getzler_kapranov98} also require the stability condition: $M((g,n))=0$
if $2g+n-2\leq 0$, which in particular excludes the unit $id\in M((0,2))$. Following Hinich-Vaintrob~\cite{hin_vaintrob02},
 we allow $M((0,2))$ to be one-dimensional spanned by the unit, but require all the other non-stable components to be trivial.}
 If one labels the inputs of $M((g_1,m))$ and $M((g_2,n))$ by respectively $m$ and $n$ elements sets $X$ and $Y$, the cyclic compositions are
 \begin{equation}\label{eq:comp1}
 \circ_{x,y}\colon M((g_1,X))\otimes M((g_2,Y))\to M((g_1{+}g_2,(X\sqcup Y)\setminus\{x,y\})),\,\, x\in X,\, y\in Y.
 \end{equation}
 The self compositions are
 \begin{equation}\label{eq:comp2}
 \circ_{x_1,x_2}M((g,X))\to M((g{+}1,X\setminus\{x_1,x_2\})),\,\, x_1,x_2\in X.
 \end{equation}
 One has an adjunction
 \[
 \MOD\colon\mathrm{CycOp}\rightleftarrows\mathrm{ModOp}\colon \mathbf{Cyc}
 \]
 between the categories of cyclic and modular operads~\cite{hin_vaintrob02}. The functor $\mathbf{Cyc}$ to any modular 
 operad assigns its $g=0$ part. Its left adjoint $\MOD$ assigns to a cyclic operad its {\it modular envelop}. 
 
 As explained by Getzler and Kapranov, the structure of a modular operad admits different twisted versions. One of them
 associated to the so called {\it cocycle} $\Det$ is of a special interest as it restricts as the same \lq\lq{}non-twisted\rq\rq{} cyclic 
 operadic structure. Roughly speaking for this twist, the composition maps~\eqref{eq:comp1} between different 
 components keep preserving
 the homological degree, while the self-compositions~\eqref{eq:comp2}  decrease the homological degree by one.
 One gets a similar adjunction
 \[
 \MOD_\Det\colon\mathrm{CycOp}\rightleftarrows\mathrm{ModOp_\Det}\colon \mathbf{Cyc}
 \]
 between the usual cyclic operads and the $\Det$-twisted modular ones.
 
 \begin{prop}[Section 3 in \cite{songhaf_tur}]\label{p:MOD}
 One has isomorpisms of symmetric sequences in chain complexes
 \[
 M(P_3^k)\simeq_{\Sigma_k} \Sigma^k\MODL((k))\otimes sign_k,\,\, k\geq 1;
 \]
 \[
 M(P_2^k)\simeq_{\Sigma_k} \Sigma\MODLdet((k)),\,\, k\geq 1.
 \]
\end{prop}

Here $\Sigma$ and $\Sigma^k$ denote the suspension and $k$-iterated suspension, respectively.
The genus of the left-hand sides corresponds to the first Betti number of the graphs.\footnote{For  readers
more familiar with the Feynman transform as defined in~\cite{getzler_kapranov98}, 
one has isomorphisms $\Sigma\mathfrak{s}\MODL\simeq \FCOMdet$, $\Sigma\mathfrak{s}\MODLdet\simeq \FCOM$,
where $\Sigma$ is the objectwise suspension, and $\mathfrak{s}$ is the modular operadic suspension, see~\cite{songhaf_tur}.}
For $d>3$, the regrading follows from the isomorphism
\begin{equation}\label{eq:Mg}
M_g(P_{d+2}^k)\simeq_{\Sigma_k}\Sigma^{2(g-1)+2k}M_g(P_d^k),
\end{equation}
where $M_g$ denotes the genus $g$ part of $M(P_{d+2}^k)$ and $M(P_{d}^k)$, respectively .

We conclude the section by recalling the well-known connection to the Vassiliev  knot 
invariants of links~\cite[Section~5.10]{getzler_kapranov98}, \cite{hin_vaintrob02}.  Consider the complex
$\empi$ for $m_1=\ldots=m_r=1$, $d=3$. Thanks to Proposition~\ref{p:MOD} and equation~\eqref{eq:empi_M(P)}, we get
\[
  \mathcal{E}_{\pi}^{1, \cdots, 1; 3} \cong\bigoplus_{k=1}^{\infty} \bigoplus_{s_1+\cdots +s_r =k} \MODL((k))_{\Sigma_{s_1}\times\ldots\times\Sigma_{s_r}}.
\]
The degree zero homology of this complex is expressed in terms of the modular envelop of the Lie operad:
\[
H_0\mathcal{E}_{\pi}^{1, \cdots, 1; 3}\cong \bigoplus_{k=1}^{\infty} \bigoplus_{s_1+\cdots +s_r =k} \mathbf{Mod}({\mathcal L}ie)((k))_{\Sigma_{s_1}\times\ldots\times\Sigma_{s_r}},
\]
which is exactly Bar Natan\rq{}s space of unitravalent graphs with colored univalent vertices, and which corresponds to the primitive
Vassiliev link invariants~\cite{barnatan95}.

\section{Low codimensional high dimensional string links}\label{s:low}
This section is devoted to Conjecture~\ref{con:low}. We explain here what supports this conjecture and why we think 
it is true.

\begin{conj}\label{con:low}
The statement~\eqref{empi_emb_formula} of Corollary~\ref{empi_emb_coro} holds for any codimension 
$\codimlow$, and any non-negative degree $*\geq 0$. 
\end{conj}

There are two things that support this conjecture. Firstly, it is true for $r=1$, which was recently shown by
Fresse, Willwacher, and the second author~\cite[Remark~19]{fr_tur_wilw}. Secondly, Theorem~\ref{t:low} below shows that it is always true 
on the level of $\pi_0$.

As for the case $r=1$, we expect two main ingredients in a possible proof of the conjecture:
the relative Hopf formality of the little discs operads, which luckily holds in any codimension $\geq 2$~\cite{tur_wilw},
and some delooping statement for $\emb$, which would say that $\emb$ is an iterated loop space of
the derived mapping space between two algebraic objects described in terms of the little discs operads. 
Unfortunately, the problem of delooping  in the case $r\geq 2$ appears to be much more difficult.  An attempt for such delooping 
in the classical case of string links (all $m_i=1$) has been done in~\cite{dwyer_hess1}. It describes an interesting possible candidate of the delooping, but unfortunately, the paper has a mistake in the proof of its main result.

\begin{thm}\label{t:low}
For $\codimlow$, $\pi_0\emb$, with the product induced by the concatenation of links, is a finitely generated abelian group.
Moreover, one has
\begin{eqnarray}\label{eq:low}
\qbb\otimes \pi_0\emb\simeq H_0\empi.
\end{eqnarray}
\end{thm}

One should mention that the fact that $\pi_i\emb$, $i>0$, are finitely generated abelian groups, easily follows
from the embedding calculus, which is always applicable when codimension $>2$~\cite{good_klein,good_weiss99}.

This theorem is a consequence of Lemmas~\ref{l:h<0} and~\ref{l:low_all}. Let us first understand the right-hand side of~\eqref{eq:low}. The complex
$\empi$ splits into a direct sum of complexes
\[
\empi=\bigoplus_{g\geq 0} \mathcal{E}_{g\pi}^{m_1, \cdots, m_r; d},
\]
where $ \mathcal{E}_{g\pi}^{m_1, \cdots, m_r; d}$ is spanned by the graphs whose first Betti number ($=$ genus) is $g$.

\begin{lem}\label{l:h<0}
For $\codimlow$, $H_{\leq 0}\empi=H_{\leq 0}  \mathcal{E}_{0\pi}^{m_1, \cdots, m_r; d}$.
\end{lem}

\begin{proof}
It is enough to show that all the graphs of positive genus have positive homological degree. One can proceed by induction.
The degree of any graph of the form  \begin{tikzpicture}[scale=.23]
 \draw (0,0) circle (1);
 \draw (-180:1) node[int]{} -- +(-1.2,0);
 \end{tikzpicture} is $d-m_i-2>0$. Any uni-trivalent hairy graph is obtained from \begin{tikzpicture}[scale=.23]
 \draw (0,0) circle (1);
 \draw (-180:1) node[int]{} -- +(-1.2,0);
 \end{tikzpicture} by a sequence of attachments of hairs (inside an edge) and edges (both endpoints on the same edge or on two different ones). The first operation creates 2 new edges and two vertices -- one external and one internal. The total degree change is 
 \[
 2(d-1)-d-m_i=d-m_i-2>0.
 \]
 The second operation creates 3 new edges and two new internal vertices. The total degree change is
 \[
 3(d-1)-2d=d-3>0.
 \]
 Finally, any other hairy graph is obtained from a uni-trivalent one by a seqence of contractions of edges. Each such contraction increases the degree by one. 
 \end{proof}

In case $r=1$, $H_*\mathcal{E}_{0\pi}^{m_1, d}$ is always one-dimensional. In case of even codimension $d-m_1$, it is 
spanned by the line graph \raisebox{.7ex}{$
\begin{tikzpicture}[scale=.7]
\draw (0,1) -- (1,1);
\end{tikzpicture}$} of degree $d-2m_1-1$; in case of odd codimension -- by the tripod $
\begin{tikzpicture}[scale=.4]
\node[int] (v) at (0,0){};
\draw (v) -- +(90:1) (v) -- ++(210:1) (v) -- ++(-30:1);
\end{tikzpicture}
$ of degree $2d-2m_1-3$.  The following result was obtained in~\cite[Corollary~20]{fr_tur_wilw} as a corollary of~\eqref{eq:low}, which was  proved there  for $r=1$.

\begin{coro}[\cite{fr_tur_wilw}]\label{cor:r=1}
For $d-m>2$, $\pi_0\overline{\mbox{Emb}}_c(\mathbb{R}^{m}, \rdbb)$ is a finitely generated abelian group of rank 
$\leq 1$. It is infinite only if either $m=2k+1$, $d=4k+3$, $k\geq 1$, or $m=4k-1$, $d=6k$, $k\geq 1$. 
\end{coro}

One has a fiber sequence
\begin{eqnarray}\label{eq:fiber-seq}
\overline{\mbox{Emb}}_c(\mathbb{R}^{m}, \rdbb)\to \mbox{Emb}_c(\mathbb{R}^{m}, \rdbb)\to \Omega^mV_m(\rdbb),
\end{eqnarray}
where $V_m(\rdbb)$ is the Stiefel manifold $SO(d)/SO(d-m)$. (By the Smale-Hirsch principle $\mbox{Imm}_c(\mathbb{R}^{m}, \rdbb)\simeq \Omega^m V_m(\rdbb)$). From~\eqref{eq:fiber-seq} we get the long exact sequence
\begin{eqnarray}\label{eq:les}
\ldots\to\pi_{m+1}V_m(\rdbb)\stackrel{j_*}{\longrightarrow}\pi_0 \overline{\mathrm{Emb}}_c(\mathbb{R}^m, \mathbb{R}^d)
\stackrel{k_*}{\longrightarrow} \pi_0  \mathrm{Emb}_c(\mathbb{R}^m, \mathbb{R}^d) \stackrel{l_*}{\longrightarrow}
\pi_mV_m(\rdbb).
\end{eqnarray}
The case $m=2k+1$, $d=4k+3$, $k\geq 1$, of the corollary corresponds to the graph-cycle  \raisebox{.7ex}{$
\begin{tikzpicture}[scale=.7]
\draw (0,1) -- (1,1);
\end{tikzpicture}$}. Geometrically it appears as the image under $j_*$ of the Euler class of $SO(d-m)$ in $V_m(\rdbb)=
SO(d)/SO(d-m)$. The case $m=4k-1$, $d=6k$, $k\geq 1$, corresponds to the graph-cycle $
\begin{tikzpicture}[scale=.4]
\node[int] (v) at (0,0){};
\draw (v) -- +(90:1) (v) -- ++(210:1) (v) -- ++(-30:1);
\end{tikzpicture}
$. Its image in $\pi_0\mathrm{Emb}_c(\mathbb{R}^m, \mathbb{R}^d)$ under $k_*$ is the Haefliger trefoil~\cite{haefliger_tref}, which is known to be trivial as immersion 
(equivalently lies in the image of $k_*$). 

To describe $H_0\mathcal{E}_{0\pi}^{m_1, \cdots, m_r; d}$ we recall some standard facts and notation. 

The operad ${{\mathcal L}ie}$ 
is cyclic, thus ${{\mathcal L}ie}(n-1)={{\mathcal L}ie}((n))$ is a $\Sigma_n$-module. The following result is well known.

\begin{lem}[\cite{levine02,habeg_pitsch03,con_schn_teich12}]\label{l:lie}
For $n\geq 1$, one has
\[
{{\mathcal L}ie}((n))\simeq_{\Sigma_n} \mathrm{ker}\left(\mathrm{Ind}^{\Sigma_n}_{\Sigma_{n-1}}{{\mathcal L}ie}(n-1)
\stackrel{q}{\to}{{\mathcal L}ie}(n)\right),
\]
where the map $q$ is described below.
\end{lem}
Elements of $\mathrm{Ind}^{\Sigma_n}_{\Sigma_{n-1}}{{\mathcal L}ie}(n-1)$ can be written in the form $\sum_{i=1}^n \alpha_i x_i\otimes
f_i(x_1,\ldots,\widehat{x}_i,\ldots,x_n)$, where $\alpha_i\in\qbb$, $f_i\in{{\mathcal L}ie}(n-1)$. The map $q$ sends
\[
q\colon x_i\otimes
f_i(x_1,\ldots,\widehat{x}_i,\ldots,x_n)\mapsto [x_i,f_i(x_1,\ldots,\widehat{x}_i,\ldots,x_n)].
\]

For any symmetric sequence $M(\bullet)$ and any graded vector space $V$, we denote by
\[
F_{M(\bullet)}(V):=\bigoplus_{k\geq 0} M(k)\otimes_{\Sigma_k}V^{\otimes k}.
\]

\begin{lem}\label{l:hempi0}
For any $\codimlow$, 
\[
H_*\mathcal{E}_{0\pi}^{m_1, \cdots, m_r; d}=\Sigma^{-d+3}F_{\lie}(\tilde H_*(\vee_{i=1}^rS^{d-m_i-2},\qbb)).
\]
\end{lem}

\begin{proof}
It is well known that the homology of the complex of trees, whose leaves are bijectively labeled by $\{1,\ldots,n\}$,
is concentrated in the top homological degree and (up to a regrading and a possible tensor with the sign representation)
is isomorphic to ${{\mathcal L}ie}((n))$, see~\cite{white}. This also follows from Proposition~\ref{p:MOD} and equation~\eqref{eq:Mg} applied for $g=0$. A careful count of the degrees and signs finishes the proof.
\end{proof}

In particular, this lemma says that $H_0 \mathcal{E}_{0\pi}^{m_1, \cdots, m_r; d}$ is isomorphic to the degree $d-3$ part
of $F_{\lie}(\tilde H_*(\vee_{i=1}^rS^{d-m_i-2},\qbb))$.

Now we want to concentrate on the left-hand side of~\eqref{eq:low}. Let $\embsphere$ be the space of smooth
spherical links $\coprod_{i=1}^r S^{m_i}\hookrightarrow S^d$.
It was shown by A.~Haefliger~\cite{haefliger} that
for $\codimlow$, $\pi_0\embsphere$ is a finitely generated abelian group, whose product is defined as follows. Given two links
$L_1,L_2\in\embsphere$, we move one inside one semisphere, and the other inside the complement semisphere. Then we connect 
the components of $L_1$ with the corresponding components of $L_2$ by very narrow tubes to form a new link. The result
as an element of $\pi_0\embsphere$ does not depend on the paths of the tubes since the complement of $L_1\coprod L_2$
is simply connected. By $\pi_0^U\embsphere$, $\pi_0^U\embnobar$, $\pi_0^U\emb$ we denote the kernel of the corresponding
map
\begin{gather*}
\pi_0\embsphere\to \prod_{i=1}^r \pi_0\mathrm{Emb}(S^{m_i},S^d);\\
\pi_0\embnobar\to  \prod_{i=1}^r \pi_0\mathrm{Emb}_c(\rbb^{m_i},\rdbb);\\
\pi_0\emb\to \prod_{i=1}^r \pi_0\overline{\mathrm{Emb}}_c(\rbb^{m_i},\rdbb).
\end{gather*}
The letter $U$ refers to \lq\lq{}unknot\rq\rq{}. Each of the groups is the subgroup of links formed by unknots.

The following lemma proves the first statement of Theorem~\ref{t:low}. Its proof is an easy exercise, which we leave to the reader.

\begin{lem}\label{l:linksU}
For $\codimlow$, one has isomorphisms of monoids (and thus of finitely generated abelian groups):
\begin{itemize}
\item $\pi_0\embnobar \simeq \pi_0 \embsphere$;
\item $\pi_0^U\emb\simeq\pi_0^U\embnobar\simeq\pi_0^U\embsphere$;
\item $\pi_0\emb\simeq \pi_0^U\emb\times \prod_{i=1}^r \pi_0\overline{\mathrm{Emb}}_c(\rbb^{m_i},\rdbb)$.
\end{itemize}
\end{lem}

By~\cite{haefliger} and Corollary~\ref{cor:r=1}, $\pi_0\embsphere$ and $\pi_0\overline{\mathrm{Emb}}_c(\rbb^{m_i},\rdbb)$
are finitely generated abelian groups, thus $\pi_0\emb$ is so as well.

Define the groups 
\begin{gather*}
\bigpibar:= \pi_{d-2}(\vee_{i=1}^r S^{d-m_i-1});\\
\biglbar:=\bigoplus_{i=1}^r\pi_{m_i}(\vee_{i=1}^r S^{d-m_i-1});\\
\bigkbar:=\mathrm{ker}(\biglbar\stackrel{q_0}{\longrightarrow}\bigpibar),
\end{gather*}
here $q_0\colon(\alpha_1\ldots \alpha_r)\mapsto \sum_{i=1}^r[\iota_i,\alpha_i]$, where $[\,.\, ,\, .\, ]$ is the Whitehead bracket,
$\iota_i\in\pi_{d-m_i-1}(\vee_{i=1}^r S^{d-m_i-1})$ is the element encoding the inclusion of the $i$-th summand
$S^{d-m_i-1}\hookrightarrow \vee_{i=1}^r S^{d-m_i-1}$. We also define their subgroups
\begin{gather*}
\bigpi:=\mathrm{ker}\left( \pi_{d-2}(\vee_{i=1}^r S^{d-m_i-1}) \to \bigoplus_{i=1}^r \pi_{d-2} S^{d-m_i-1}\right);\\
\bigla:=\bigoplus_{i=1}^r\mathrm{ker}\left(\pi_{m_i}(\vee_{i=1}^r S^{d-m_i-1})\to \pi_{m_i} S^{d-m_i-1}\right);\\
\bigk:=\mathrm{ker}(\bigla\stackrel{q_0|_{\Lambda}}{\longrightarrow}\bigpi).
\end{gather*}

\begin{lem}\label{l:low_all}
For $\codimlow$, 
\begin{eqnarray}\label{eq:low_all}
\qbb\otimes\pi_0\emb\simeq\qbb\otimes\bigkbar\simeq H_0 \mathcal{E}_{0\pi}^{m_1, \cdots, m_r; d}.
\end{eqnarray}
\end{lem}

\begin{proof}
We  prove below the second isomorphism in~\eqref{eq:low_all} for any $r\geq 1$. Thus, by Corollary~\ref{cor:r=1} it\rq{}s
enough to show the first isomorphism only for $r\geq 2$. In~\cite{haefliger}, A.~Haefliger develops a homotopical approach
to compute $\pi_0\embsphere$. Using this method it was shown in~\cite[Lemma~1.3]{cr_fer_skop} that
\[
\qbb\otimes \pi_0^U\embsphere=\qbb\otimes \bigk.
\]
Applying Lemma~\ref{l:linksU}, we get
\begin{multline*}
\qbb\otimes\pi_0\emb\simeq\qbb\otimes \pi_0^U\emb\oplus\bigoplus_{i=1}^r
 \qbb\otimes\pi_0\overline{\mathrm{Emb}}_c(\rbb^{m_i},\rdbb)\\
 \simeq \qbb\otimes\bigk\oplus\bigoplus_{i=1}^r \qbb\otimes\overline{K}^{m_id}
 \simeq \qbb\otimes\bigkbar.
 \end{multline*}
 The last isomorphism is due to the
 splittings
\begin{gather*}
\bigpibar=\bigpi\oplus \bigoplus_{i=1}^r \pi_{d-2} S^{d-m_i-1} = \bigpi\oplus\bigoplus_{i=1}^r \overline{\Pi}^{m_id};\\
\biglbar=\bigla\oplus\bigoplus_{i=1}^r \pi_{m_i} S^{d-m_i-1} = \bigla\oplus\bigoplus_{i=1}^r \overline{\Lambda}^{m_id},
\end{gather*}
and also the fact that  $q_0$ sends each summand of $\biglbar$ to the corresponding summand of $\bigpibar$.  

To see the second isomorphism of~\eqref{eq:low_all}, we notice that
\[
\qbb\otimes \pi_*(\vee_{i=1}^r S^{d-m_i-1})\simeq \Sigma F_{{{\mathcal L}ie}(\bullet)}(\tilde H_*(\vee_{i=1}^r S^{d-m_i-2},\qbb)).
\]
Thus $\qbb\otimes\bigpibar$ is the degree $d-3$ part of $F_{{{\mathcal L}ie}(\bullet)}(\tilde H_*(\vee_{i=1}^r S^{d-m_i-2},\qbb))$. Similarly, one can see that $\qbb\otimes\biglbar$ is the degree $d-3$ part of
$F_{\mathrm{Ind}^{\Sigma_\bullet}_{\Sigma_{\bullet-1}}\lieminus}(\tilde H_*(\vee_{i=1}^r S^{d-m_i-2},\qbb))$. Thus
$\bigkbar$ is the degree $d-3$ part of
\[
\mathrm{ker}\left(F_{\mathrm{Ind}^{\Sigma_\bullet}_{\Sigma_{\bullet-1}}\lieminus}(\tilde H_*(\vee_{i=1}^r S^{d-m_i-2},\qbb))
\stackrel{F_q}{\longrightarrow}
 F_{{{\mathcal L}ie}(\bullet)}(\tilde H_*(\vee_{i=1}^r S^{d-m_i-2},\qbb))\right),
\]
which by Lemmas~\ref{l:lie} and~\ref{l:hempi0} is the same as $H_0 \mathcal{E}_{0\pi}^{m_1, \cdots, m_r; d}$.

\end{proof}

\section{The Koszul complexes} \label{generating_function_homology}

By Theorems~\ref{rational_homology_thm} and~\ref{rational_homotopy_thm}, the rational homology and homotopy of $\embthm$ is described  in terms  of derived maps of right $\Omega$-modules. In this section we  exhibit explicit complexes computing the (co)homology and homotopy of $\embthm$ by considering the projective resolution of the source $\Omega$-modules 
in~\eqref{final_splitting_H} and~\eqref{final_splitting_pi}.  We call these complexes {\it Koszul} as they are constructed using operadic language and the Koszul duality between the commutative and Lie operads. At the end of the section we show that the Koszul complexes can  also be used to understand rationally the stages of the Goodwillie-Weiss towers for $\emb$.

%


For $s_1, \cdots, s_r \geq 0$, recall the definition of the right $\Omega$-module $Q_{s_1 \cdots s_r}^{m_1 \cdots m_r}$ (it was introduced in  (\ref{qs1sr})).  

\begin{eqnarray} \label{ps1sr}
Q_{s_1 \cdots s_r}^{m_1 \cdots m_r}(k) = \left\{ \begin{array}{lll}
                                 0 & \mbox{if} & k \neq s_1+\cdots+s_r \\
																\mbox{Ind}^{\Sigma_k}_{\Sigma_{s_1} \times \cdots \times \Sigma_{s_r}} \widetilde{H}_*(S^{s_1m_1+\cdots+s_rm_r} ; \mathbb{Q}) & \mbox{if} & k = s_1+\cdots+s_r
                                \end{array} \right.
\end{eqnarray}

\begin{prop} \label{rational_homology_prop2}
For $d > 2 \mathrm{max}\{m_i| \ 1 \leq i \leq r\} + 1$, one has quasi-isomorphisms 
\begin{eqnarray} \label{rational_homology_quasi-iso2}
C_*(\embthm, \otimes \mathbb{Q}) \simeq \bigoplus_{s_1, \cdots, s_r, t \geq 0} \underset{\Omega}{\mathrm{hRmod}}  (Q^{m_1\cdots m_r}_{s_1 \cdots s_r}, \widehat{H}_{t(d-1)} (C(\bullet, \rdbb); \mathbb{Q}))
\end{eqnarray} 
\begin{eqnarray} \label{rational_homotopy_quasi-iso2}
\ \pi_*\embthm\otimes\qbb\simeq \bigoplus_{s_1, \cdots, s_r, t \geq 0} \underset{\Omega}{\mathrm{hRmod}}  (Q^{m_1\cdots m_r}_{s_1 \cdots s_r}, \widehat{\pi}_{t(d-2)+1} C(\bullet, \rdbb)\otimes \mathbb{Q})
\end{eqnarray} 
\end{prop}

\begin{proof}
This is a direct consequence of  Theorems~\ref{rational_homology_thm} and~\ref{rational_homotopy_thm} and the fact that both the source and target $\Omega$-modules in~\eqref{rational_homology_iso} and~\ref{rational_homotopy_iso} split into a direct
sum/product.
 The only thing we need to understand is why the product can be replaced by the direct sum. This follows for example, from Remark~\ref{rmq:finite_dim}.
\end{proof}


\subsection{The Koszul resolution  and Koszul dual of  
$Q_{s_1 \cdots s_r}^{m_1 \cdots m_r}$} \label{koszul_dual_right_module}

In this section we will find a cofibrant replacement (the model structure we use in the category $\underset{\Omega}{\mbox{Rmod}}$ is the one whose weak equivalences and fibrations are levelwise), denoted by $C\qms$, of $\qms$. We cal this cofibrant replacement 
{\it Koszul resolution} of $\qms$. The idea of this construction is to use the fact that the right $\Omega$-module structure is
the same as the right module structure over the commutative non-unital operad ${\mathcal C}om_+$ and then use the Koszul duality
between the commutative and Lie  operads.
We will follow exactly the same steps as in \cite[Section 5 and Section 6]{aro_tur13}, where this was done for $r=1$.  
  As an $\Omega$-module, $C\qms$ is freely generated by a symmetric sequence that we denote by $K\qms$ and we call it
  {\it Koszul dual} of $\qms$.\footnote{One should not confuse the terms {\it Koszul complexes}, which are  complexes computing the (co)homology or homotopy of embedding spaces, {\it Koszul resolutions}, which refer to a cofibrant replacement of $\Omega$-modules, and  {\it Koszul duals} which refer to symmetric  sequences (in fact to right $co{\mathcal L}ie[1]$-comodules).} We will also explicitly describe $C \qms$ (and then explicitly describe $K \qms$) in terms of  spaces of colored forests. 

Let us start with an explicit description of $C\qms$. For $r =1$, $\qms$ is exactly the right $\Omega$-module $Q^m_s$ that appears in ~\cite[Section~5]{aro_tur13}, where $CQ^m_s$ is explicitly described in terms of forests. Here we will describe $C \qms$ in a similar way, except that in our case we will get "colored forests" with $\{1, \cdots, r\}$ as the set of colours. 

Recall first two standard notations. The first one is the operation "$\circ$" (explicitly defined in~\cite[Chapter 5]{lod_val12}) in the category of symmetric sequences in chain complexes~\cite[\S~5.1]{lod_val12}. Intuitively, if $M = \{M(n)\}_{n \geq 0}$ and $N = \{N(n)\}_{n \geq 0}$ are two symmetric sequences, elements in $M \circ N (n)$ are of the form $x (y_1, \cdots, y_k)$, where $x \in M(k)$, $y_i \in N(i_k)$, and $i_1+\cdots+i_k =n$ (one should think about an element in $M(k)$ as an operation of arity $k$).  
In fact one can define "$\circ$" in the category of symmetric sequences in any cocomplete symmetric monoidal category. This operation defines a monoidal structure on symmetric sequences. The unit for this structure is the sequence $\mathbf{1}$, defined as $\mathbf{1}(1)$ is the unit and  every $\mathbf{1}(n)$, $n\neq 1$,  is the initial element.  Monoids with respect to $\circ$ are operads. The second notation is that of suspension of a symmetric sequence. Let $M = \{M(n)\}_{n \geq 0}$ be a symmetric sequence in chain complexes. The \textit{operadic  suspension}, denoted by $M[1]$, of $M$ is a symmetric sequence defined by 
\[M[1](n) = \Sigma^{n-1} M(n) \otimes \mbox{sign}_n,\]
where $\Sigma^{n-1}$ is the usual $(n-1)-$ suspension, and $\mbox{sign}_n$ is the sign representation of $\Sigma_n$. The operadic suspension of an operad is again an operad because of the identities $\mathbf{1}[1]= \mathbf{1}$ and $(M \circ N)[1] = M[1] \circ N[1]$.

Given a  right module $\mathcal{R}$ over any Koszul operad $\mathcal{K}$,  Fresse defines in \cite{fress08} a cofibrant replacement, denoted by $C\mathcal{R}$, of $\mathcal{R}$. If $\mathcal{K}^{\rotatebox[origin=c]{180}{!}}$ denotes the cooperad Koszul dual to $\mathcal{K}$, as a symmetric sequence $C\mathcal{R} = \mathcal{R} \circ \mathcal{K}^{\rotatebox[origin=c]{180}{!}}[1] \circ \mathcal{K}$. Its differential consists of three summands. The first one is induced by the internal differential on  $\mathcal{R}$. The second one arises from the right $\mathcal{K}$ module structure of $\mathcal{R}$. The third one arises from the left  $\mathcal{K}$ module structure of  $\mathcal{K}$. 
One can see that $C\mathcal{R}$ as a right $\mathcal{K}$ module is freely generated by the symmetric sequence $K\mathcal{R}
= \mathcal{R} \circ \mathcal{K}$. This sequence $K\mathcal{R}$ with its natural differential is sometimes called {\it  Koszul dual}
of $\mathcal{R}$. It naturally carries the structure of a right $\mathcal{K}^{\rotatebox[origin=c]{180}{!}}[1]$ 
comodule.\footnote{Here, Koszul duality refers to  the equivalence of the homotopy category of right $\mathcal{K}$-modules
and that of right $\mathcal{K}^{\rotatebox[origin=c]{180}{!}}[1]$-comodules.}
  Applying Fresse's construction to $\qms$, viewed as a right module over ${{\mathcal C}om}_+$ (here the action of  ${{\mathcal C}om}_+$ on $\qms$ is trivial), and by noticing that  the cooperad Koszul dual  of ${{\mathcal C}om}_+$ is ${co{\mathcal L}ie}$
  (which is  is the dual of the operad $\mathcal{L}ie$ of Lie algebras),  we get 
\begin{eqnarray} \label{qms_defn} 
C \qms = \qms \circ {co{\mathcal L}ie} [1] \circ {{\mathcal C}om}_+.
\end{eqnarray}
 For obvious reasons only the third summand in the differential will be non-trivial. The equation (\ref{qms_defn}) allows us to explicitly describe the right $\Omega$-module $C\qms$. Recall that the cooperad ${co{\mathcal L}ie}$ can be described as a space of trees modulo Arnold (3 term) relations~\cite{sinha}.  Let $A$ be a finite set of cardinal $k$. Then $C \qms (A)$ or simply $C \qms (k)$ is a chain complex spanned by the following set $\mathcal{F}_{s_1 \cdots s_r}^{m_1 \cdots m_r}(k)$ of \textit{oriented colored  forests}: 
\begin{enumerate}
\item[$\bullet$] A \textit{vertex} in such forest $F$ is a nonempty subset of $A$. The set of vertices forms a partition $\coprod_i A_i$ of $A$. 
\item[$\bullet$] Vertices of $F$, labeled by subsets $A_i$, are connected by edges so that no cycles appear, or, in other words, the obtained graph is a forest.  
\item[$\bullet$] Such forest $F$ must consist of exactly $s_1 + \cdots +s_r$ trees. We denote by $T_F$ the set of its connected
components.
\item[$\bullet$] A \textit{colored forest} is a couple $(F, f)$, where $F$ is a forest as above, and $f \colon T_F \lra \{1, \cdots, r\}$ is a  map (called coloring) from the set of trees in $F$ to the set of colours $\{1, \cdots, r\}$ such that for all $i = 1, \cdots, r$, $|f^{-1}(i)| = s_i$. 
\item[$\bullet$] An \textit{oriented colored forest} or an element in $\mathcal{F}_{s_1 \cdots s_r}^{m_1 \cdots m_r}(k)$  is a triple $(F, f, O)$, where $(F,f)$ is a colored forest, and $O = T_F \cup E_F$ (here $E_F$ denotes the set of all edges in $F$) is endowed with a total order $\leq$. The set $O$ is the so called \textit{orientation set}.   
\end{enumerate}

We define a degree of an element of $\mathcal{F}_{s_1 \cdots s_r}^{m_1 \cdots m_r}(k)$. Let  $(F, f, O)$ be any element in $\mathcal{F}_{s_1 \cdots s_r}^{m_1 \cdots m_r}(k)$.  Each edge of $F$ is of degree~$1$. A tree  $T$ in $F$ is of degree $m_i$ if $T \in f^{-1}(i)$. The total degree of $(F, f, O)$ is then
\begin{eqnarray} \label{total_degree}
\mbox{deg}(F, f, O) = |E_F| + \underset{1 \leq i \leq r}{\sum} s_im_i.
\end{eqnarray} 

 Let $\qbb [\mathcal{F}_{s_1 \cdots s_r}^{m_1 \cdots m_r}(k)]$ denote the graded vector space over $\qbb$ spanned by $\mathcal{F}_{s_1 \cdots s_r}^{m_1 \cdots m_r}(k)$. We consider two relations in $\qbb [\mathcal{F}_{s_1 \cdots s_r}^{m_1 \cdots m_r}(k)]$. The first one is the sign relation denoted by $sign$. More precisely, let $(F_1, f_1, O_1)$ and $(F_2, f_2, O_2)$ be two forests.  We write $(F_1, f_1, O_1) = \pm (F_2, f_2, O_2)$ if $(F_1, f_1, O_1)$  differs from $(F_2, f_2, O_2)$ only by reordering of the orientation set.  The sign is the Koszul sign of permutation taking into account the degrees of the elements.  The second relation is the Arnold relation, denoted by $Arnold$, which is of the form

%
%

\begin{center}
\begin{tabular}{c}
\includegraphics[scale = 1.0]{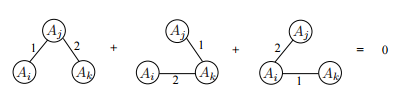} \\
Arnold's relation
\end{tabular}
\end{center}

This relation is local, that is, all other vertices and edges are the same in each summand. Finally there is  a differential in $\qbb [\mathcal{F}_{s_1 \cdots s_r}^{m_1 \cdots m_r}(k)]$ defined as follows. For an edge $e \in 	E_F$, let $(F/e, f, O\backslash \{e\})$ be the forest obtained from $(F, f, O)$ by contracting $e$. Recall that every vertex in $F$ is labeled by some subset $A_i\subset A$.
When an edge between two sets $A_i$ and $A_j$ is contracted, the new vertex is assigned the set $A_i\sqcup A_j$. Define a sign $(-1)^e \in \{-1,1\}$ by $(-1)^e = (-1)^{\sum_{x < e} |x|}.$ The differential of $(F, f, O)$ is defined by 
\[d(F, f, O) = \sum_{e \in E_F} (-1)^e (F/e, f, O \backslash \{e\})\] 

%
%

 The sequence $\left\{\frac{\qbb [\mathcal{F}_{s_1 \cdots s_r}^{m_1 \cdots m_r}(k)]}{sign,\ Arnold} \right\}_{k \geq 0}$ is equipped with a structure of a right $\Omega$-module defined as follows. Let $g \colon A' \lra A$ be a surjective map from a finite set $A'$ (of cardinal $k'$) to $A$ (of cardinal $k$). Define 
$$g^* \colon \frac{\qbb [\mathcal{F}_{s_1 \cdots s_r}^{m_1 \cdots m_r}(A)]}{sign, Arnold} \lra \frac{\qbb [\mathcal{F}_{s_1 \cdots s_r}^{m_1 \cdots m_r}(A')]}{sign, Arnold}$$
as the map sending a forest $(F, f, O)$ to the forest $g^*(F, f, O)$ obtained by replacing each vertex  of $F$ labeled by a subset $A_i \subseteq A$ by a vertex labeled by $A'_i=g^{-1}(A_i)$. The orientation and the coloration of $g^*(F, f, O)$ remain unchanged. It is easy to see that 
the right $\Omega$-module $C\qms$ defined by~\eqref{qms_defn} is exactly $\frac{\qbb [\mathcal{F}_{s_1 \cdots s_r}^{m_1 \cdots m_r}]}{sign, Arnold}$. We reiterate, one just need to recall and apply the description of ${co{\mathcal L}ie}$ in terms
of spaces of trees modulo Arnold relations~\cite{sinha}. 


Now we will describe the Koszul dual $K \qms$ of $\qms$. It is defined as 
\[ K \qms = \qms \circ {co{\mathcal L}ie}[1]. \] 
So $C \qms$, viewed as a right ${{\mathcal C}om}_+$ module (we can view it like this because of (\ref{qms_defn}) above), is freely generated by the symmetric sequence $K \qms$. By  definition we have an inclusion $K \qms \subseteq C \qms$, and generators of $K \qms$ are forests $(F, f, O)$ whose all vertices are singletons.  For each $k$, $K \qms (k)$ is concentrated in the single homological degree 
\[ \sum_{i=1}^r m_is_i + \sum_{i=1}^r (k_i-s_i) = \sum_{i=1}^r (s_i(m_i-1)+k_i)=
\sum_{i=1}^r s_i(m_i-1) +k, \] 
where $k_i$ is the number of vertices colored by $i$ (so $k= k_1 + k_2 + \cdots + k_r$). 

The Koszul dual $K \qms$ is isomorphic to a symmetric sequence related to the homology of configuration spaces as we will see in Proposition~\ref{koszul_dual_prop} below. Before stating this proposition we recall some notation.  For a pointed topological space $X$, and for $p \geq 1$, let $\Delta^pX$ denote the "fat diagonal" in $X^{\wedge p} = \underbrace{X \wedge \cdots \wedge X}_{p}$ (recall that the symbol "$\wedge$" is that of the smash product). More precisely, we have 
$$\Delta^pX = \underset{1 \leq i \neq j  \leq p}{\bigcup} \{(x_1, \cdots, x_p) \in X^{\wedge p}| \ x_i =x_j\}.$$

\begin{prop} \label{koszul_dual_prop}
For $s_1, \cdots, s_r \geq 0$, and for $k \geq 0$ there is an isomorphism of symmetric sequences

\begin{eqnarray} \label{koszul_dual_iso}
 K \qms (\bullet) \cong \underset{\bullet = k_1+\cdots + k_r}{\bigoplus} \left( \mathrm{Ind}^{\Sigma_k}_{\Sigma_{\overline k}}  \bigotimes_{i=1}^r \widetilde{H}_{(m_i-1)s_i +k_i}(S^{m_ik_i} / \Delta^{k_i}S^{m_i}, \qbb) \right),  
\end{eqnarray}
where 
 $\Sigma_{\overline k} = \Sigma_{k_1} \times \cdots \times \Sigma_{k_r}$. 

\end{prop}

\begin{proof}
The first thing to note is the following. For $1 \leq i \leq r$ one has
\begin{eqnarray} \label{htilde_hbar}
\widetilde{H}_*(S^{m_ik_i} / \Delta^{k_i} S^{m_i}, \qbb) = \overline{H}_*(C(k_i, \rbb^{m_i}), \qbb),
\end{eqnarray}
where $\overline{H}_*(-)$ is the locally compact singular homology functor.  The rest of the proof uses the Poincar\'e duality between $\overline{H}_*(C(k, \rbb^m), \qbb)$ and $H^{mk-*}(C(k, \rbb^m), \qbb)$, and also the fact that the latter cohomology admits a similar description in terms of forests (see \cite{arnold69,Coh}).
\end{proof}

\begin{rmq}
There is a more natural topological way to prove~\eqref{koszul_dual_iso}. In \cite[Section 11]{aro_tur12} one can find a general statement, saying that for any pointed space $X$, the Koszul dual of $\widetilde{C}(X^{\wedge \bullet})$ is equivalent to $\widetilde{C}_*(X^{\wedge \bullet} / \Delta^{\bullet} X)$. In our case $X= \vee_{i=1}^r S^{m_i}$, because of the formality  the singular chains can be replaced by the homology:
\begin{multline*}
K\widetilde{H}_*((\vee_{i=1}^r S^{m_i})^{\wedge \bullet}, \qbb)\cong
\widetilde{H}_*\left( ( \vee_{i=1}^r S^{m_i})^{\wedge \bullet}/ \Delta^\bullet( \vee_{i=1}^r S^{m_i}),\qbb\right)\\
\cong \bigoplus_{k_1+\ldots + k_r =\bullet} 
\mathrm{Ind}^{\Sigma_k}_{\Sigma_{\overline k}} \widetilde{H}_*(\wedge_{i=1}^r (S^{m_ik_i}/\Delta^{k_1}S^{m_i}),\qbb),
\end{multline*}
where as before $\Sigma_{\overline k}=\Sigma_{k_1}\times\ldots\times \Sigma_{k_r}$. 
\end{rmq}


\subsection{Koszul complex computing the rational homology of $\emb$} \label{graph_complex_homology}\label{ss:koszul_homol}

In this subsection we describe a complex computing the rational homology of $\emb$.

Recall the following definition. 

\begin{defn} \label{otimes_hat_defn}
Let $V = \{V(n)\}_{n \geq 0}$ and $W = \{W(n)\}_{n \geq 0}$ be two symmetric sequences. Define a new symmetric sequence $V \widehat{\otimes} W$ by 
\[V \widehat{\otimes} W (n) = \bigoplus_{p+q =n} \mathrm{Ind}^{\Sigma_n}_{\Sigma_p \times \Sigma_q} V(p) \otimes W(q).\] 
\end{defn}

In practice we need to deal with multigraded vector spaces. Besides the usual homological degree, they will have the 
Hodge multi-grading $(s_1,\ldots,s_r)$ and grading by complexity $t$. As usual when we take tensor product all the degrees
get added.


\begin{prop} \label{total_complex_homology}
For $d > 2 \mathrm{max}\{m_i| \ 1 \leq i \leq r\} + 1$, there is a quasi-isomorphism 

\begin{multline} \label{total_complex_quasi-iso}
C_*(\embthm, \mathbb{Q}) \simeq \left( \underset{k \geq 0}{\bigoplus} \mathrm{hom}_{\Sigma_k} \left( \underset{1 \leq i \leq r}{\widehat{\otimes}} \overline{H}_*(C(\bullet, \rbb^{m_i}), \qbb)(k), \widehat{H}_*(C(k, \rdbb), \qbb) \right), \partial \right)\\
\simeq \left( \underset{k_1,\ldots,k_r}{\bigoplus} \mathrm{hom}_{\Sigma_{k_1}\times\ldots\times\Sigma_{k_r}}
\left(\underset{1\leq i\leq r}\otimes  \overline{H}_*(C(k_i, \rbb^{m_i}) , \widehat{H}_*(C(k_1+\ldots+k_r, \rdbb), \qbb) \right)  ,\partial\right).
\end{multline}

\end{prop}

The differential $\partial$ here is defined similarly to the $r=1$ case \cite[Subsection~5.2]{aro_tur13}. We won\rq{}t describe it explicitly here. However, in next Subsection~\ref{graph_complex_cohomology} we describe explicitly the dual complex computing the cohomology of the space of links.

\begin{proof} It follows  from Proposition~\ref{rational_homology_prop2} and Proposition~\ref{koszul_dual_prop}. The idea is to replace each term in the right hand side of (\ref{rational_homology_quasi-iso2}) by $\underset{\Omega}{\mbox{Rmod}}  (CQ^{m_1\cdots m_r}_{s_1 \cdots s_r}, \widehat{H}_{t(d-1)} (C(\bullet, \rdbb); \mathbb{Q}))$ (since, in Section~\ref{koszul_dual_right_module}, we have seen that $C \qms$ is a cofibrant replacement of $\qms$), and use the fact that $C \qms$ is freely generated by $K \qms$. The rest of the proof is very similar to \cite[Section 5.2]{aro_tur13}. 
\end{proof}

\subsection{Graph-complex computing the rational cohomology of $\emb$} \label{graph_complex_cohomology}

In this subsection we will explicitly describe a graph-complex, denoted $HH^{m_1,\cdots,m_r; d}$, computing the rational cohomology of $\emb$ and which is dual to~ (\ref{total_complex_quasi-iso}).  The graph-complex $HH^{m;d}$ was described in~ \cite[Section 5.3]{aro_tur13}.
In a similar way  we will describe here $HH^{m_1,\cdots,m_r; d}$ except that our complex will be colored by the set $\{1, \cdots, r\}$ of components of our links. 


By taking the dual of (\ref{total_complex_quasi-iso}), and recalling "$\widehat{\otimes}$" from Definition~\ref{otimes_hat_defn}, we have a quasi-isomorphism 
\begin{align}
& C^*(\emb, \qbb)  \\
 & \simeq \left( \underset{k \geq 0}{\bigoplus}\,\, 
 \underset{k = k_1 + \cdots + k_r} {\bigoplus} \left( \mbox{Ind}_{\Sigma_{k_1} \times \cdots \times \Sigma_{k_r}}^{\Sigma_k} \overline{H}_{-*} C(k_1, \rbb^{m_1}) \otimes \cdots \otimes \overline{H}_{-*} C(k_r, \rbb^{m_r}) \right) \otimes_{\Sigma_k} \widehat{H}^* C(k, \rdbb)\, ,\, d\, \right) \label{dual_complex0}\\
 &\simeq \left( \underset{k \geq 0}{\bigoplus}\,\, \underset{k = k_1 + \cdots + k_r} {\bigoplus} \left( \overline{H}_{-*} C(k_1, \rbb^{m_1}) \otimes \cdots \otimes \overline{H}_{-*} C(k_r, \rbb^{m_r}) \right) \otimes_{\Sigma_{k_1}\times\ldots\times\Sigma_{k_r}} \widehat{H}^* C(k, \rdbb)\, ,\, d\, \right)
  \label{dual_complex}
\end{align}
By using the Poincar\'e duality, one has 
\begin{eqnarray} \label{poincare_duality}
\overline{H}_{-*} C(k_i, \rbb^{m_i}) \cong {H}^{*+m_ik_i} C(k_i, \rbb^{m_i}) \otimes (\mbox{sign}_{k_i})^{\otimes m_i}, 1 \leq i \leq r. 
\end{eqnarray}

The equations (\ref{dual_complex}) and (\ref{poincare_duality}) tell us that the cochain complex $C^*(\emb, \qbb)$ is quasi-isomorphic to a complex whose building blocks are the cohomology groups of configuration spaces, which are well known (see for example \cite{arnold69,Coh}). Recall that $H^*(C(p, \rbb^n))$ is an   algebra generated by elements $g_{ij}, 1 \leq i \neq j \leq n$, of degree $n-1$, modulo relations 
\[g_{ij} = (-1)^n g_{ji}, \,\, g^2_{ij} =0, \ \mbox{and the Arnold relation:} \ g_{ij}g_{jk} + g_{jk}g_{ki} + g_{ki}g_{ij} = 0.\]
This cohomology algebra 	admits a nice description in terms of  spaces of forests. To every non-zero monomial, one can assign a forest putting an edge from vertex $i$ to vertex $j$ for each factor $g_{ij}$. Using this description, and using (\ref{poincare_duality}) and (\ref{dual_complex}), we have the graph-complex $HH^{m_1,\cdots,m_r; d}$ (computing $H^*(\emb, \qbb)$), which is described as follows. It is a graded vector space  spanned by some set, denoted $\mathcal{G}$, of graphs,
modulo some relations. An element in $\mathcal{G}$ is a quadruple $(G, f^v_G, f^e_G,O_G)$, where 
\begin{enumerate}
\item[$\bullet$]  $G = (V_G, E_G)$ is a graph, possibly disconnected, in which $V_G$ is a finite set of non-labeled vertices (vertices are non-labeled because the last tensor product in (\ref{dual_complex0}) is taken over $\Sigma_k$), and $E_G$ is a finite set of oriented edges; 
\item[$\bullet$]  $f^v_G \colon V_G \lra \{1, \cdots, r\}$ is a function (possibly non-surjective) called a \textit{coloration} of vertices,
and $f^e_G\colon E_G \lra  \{0,1,\cdots,r\}$ is a similar \textit{coloration} of edges . We say that elements in $(f^v_G)^{-1}(i)$ are vertices  colored by $i$. Similarly for edges, elements in  $(f^e_G)^{-1}(i)$, $1\leq i\leq r$,  are said {\it colored} (by $i$), while  
the elements in $(f^e_G)^{-1}(0)$ are said {\it non-colored} edges. The latter edges (corresponding to generators of $H^*(C(k, \rdbb), \qbb)$) can connect any two vertices, while colored edges (corresponding to generators of $H^*(C(k_i, \rbb^{m_i}), \qbb)$) connect only two vertices having the same color (see Figure~\ref{graph_colored} in which $r=2$). 
\begin{figure}[!ht]
\centering
\includegraphics[scale=0.6]{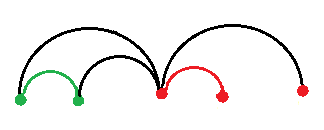}
\caption{A graph in $\mathcal{G}$} \label{graph_colored}
\end{figure}

There are two restrictions on such graphs.
If we remove all colored edges, the resulting graph is a forest in which each connected component has at least two vertices. On the other hand, if we remove all non-colored edges, the resulting graph is a forest with any type of connected components;
\item[$\bullet$]   $O_G$ is a total order of the orientation  set of $G$, where the latter one is defined in two equivalent ways. One switches from one description to another by applying the Poincar\'e duality~\eqref{poincare_duality}. The first way  naturally reminds the Bott-Taubes integration. The second way is algebrically more natural from the point of view of Koszul duality. In the first way, the orientation set is the union of the set of 
  \begin{enumerate}
   \item[-] non-colored oriented edges (considered as elements of degree $d-1$),
   \item[-] colored oriented edges (an  edge colored by $i$ is of degree $m_i-1$), and
   \item[-] vertices (a vertex colored by $i$ is of degree $-m_i$). 
  \end{enumerate}
  In the second way, it is the union of the set of 
   \begin{enumerate}
   \item[-] non-colored oriented edges (considered as elements of degree $d-1$),
   \item[-] colored non-oriented edges (of degree $-1$), and
   \item[-] connected components with respect to colored edges (such a component, whose edges are colored by $i$, is of degree $-m_i$).  
   \end{enumerate}
\end{enumerate}
Let $\mbox{Vect} [\mathcal{G}]$ denote the vector space over $\qbb$ spanned by $\mathcal{G}$, modulo the following relations. 
\begin{enumerate}
\item[-] $G=\pm G\rq{}$ if $G$ differs from $G$ only by the order of the orientation set. The sign $\pm$ is the Koszul sign of 
the corresponding permutation taking into account the degrees of the elements;
\item[-] $G = (-1)^d G'$ if $G$ differs from $G'$ only by the orientation of a non-colored edge;
\item[-] $G = (-1)^{m_i} G'$ if $G$ differs from $G'$ only by the orientation of an edge colored by $i$;
\item[-] The Arnold relations with respect to both colored and non-colored edges. 
\end{enumerate}
The third relation is valid only when we consider the first definition of an orientation set.  
 The differential, denoted $\delta$, in $\mbox{Vect} [\mathcal{G}]$  is the sum of contractions of colored edges. In the second description of orientation sets, the sign is defined by pulling on the first place of the orientation set the colored edge which is contracted and then removing it from this set. For the fist description of orientation sets, the sign is slightly more complicated. We need to pull on respectively first and second places of the orientation set the contracted colored edge and the vertex, which is its source (to recall this edge is oriented), and then remove both of them.

\begin{defn}
$$HH^{m_1,\cdots,m_r; d} = (\mbox{Vect} [\mathcal{G}],\delta).$$
\end{defn}

The following result is an immediate consequence of (\ref{dual_complex0}), (\ref{dual_complex}), and (\ref{poincare_duality}). 

\begin{thm} \label{complex_cohomology_thm} 
For $\codim$, we have an isomorphism
\[H^*(HH^{m_1,\cdots,m_r; d}) \cong H^*(\embthm, \qbb).\]
\end{thm}

\begin{rmq} The disjoint union of graphs defines a product in $HH^{m_1,\cdots,m_r; d}$. Moreover, the latter complex  is a free graded commutative algebra generated by its subcomplex $HH^{m_1,\cdots,m_r; d}_\pi$ 
of connected graphs (with respect to both colored and non-colored edges). Therefore the homology
of $HH^{m_1,\cdots,m_r; d}_\pi$ is naturally isomorphic to the space dual to the rational homotopy of 
$\embthm$.
\end{rmq}

\subsection{Koszul complex computing the rational homotopy of $\emb$} \label{graph_complex_homotopy}

Similarly to Proposition~\ref{total_complex_homology}, we get:

\begin{prop} \label{total_complex_homotopy}
For $d > 2 \mathrm{max}\{m_i| \ 1 \leq i \leq r\} + 1$, there is a quasi-isomorphism 

\begin{multline} \label{total_complex_quasi-iso_pi}
\pi_*\embthm \otimes \mathbb{Q} \simeq \left( \underset{k \geq 0}{\bigoplus} \mathrm{hom}_{\Sigma_k} \left( \underset{1 \leq i \leq r}{\widehat{\otimes}} \overline{H}_*(C(\bullet, \rbb^{m_i}), \qbb)(k), \widehat{\pi}_*C(k, \rdbb)\otimes\qbb \right), \partial \right)\\
\simeq \left( \underset{k_1,\ldots,k_r}{\bigoplus} \mathrm{hom}_{\Sigma_{k_1}\times\ldots\times\Sigma_{k_r}}
\left(\underset{1\leq i\leq r}\otimes  \overline{H}_*(C(k_i, \rbb^{m_i}) , \widehat{\pi}_*C(k_1+\ldots+k_r,\rdbb)\otimes \qbb \right)  ,\partial\right).
\end{multline}
\end{prop}

The differential $\partial$ here is defined similarly to the case of the homology complex. For  the $r=1$ case, 
see~\cite[Subsection~5.2]{aro_tur13}. 

\begin{proof}
Analogous to the proof of Proposition~\ref{total_complex_homology}.
\end{proof}

\subsection{Truncations of the Koszul complexes and the Goodwillie-Weiss tower}\label{ss:truncation_koszul}

Contrary to the hairy graph-complexes, the Koszul complexes are well suited to determine rationally the Goodwillie-Weiss
approximations for $\emb$.

\begin{thm}\label{th:trunc}
For $\codimlow$, one has
\begin{multline}\label{eq:trunc_homology}
T_nC_*(\embthm, \mathbb{Q}) \simeq \left( \underset{0\leq k \leq n}{\bigoplus} \mathrm{hom}_{\Sigma_k} \left( \underset{1 \leq i \leq r}{\widehat{\otimes}} \overline{H}_*(C(\bullet, \rbb^{m_i}), \qbb)(k), \widehat{H}_*(C(k, \rdbb), \qbb) \right), \partial \right)\\
\simeq \left( \underset{{k_1,\ldots,k_r}\atop{k_1+\ldots +k_r\leq n}}{\bigoplus} \mathrm{hom}_{\Sigma_{k_1}\times\ldots\times\Sigma_{k_r}}
\left(\underset{1\leq i\leq r}\otimes  \overline{H}_*(C(k_i, \rbb^{m_i}) , \widehat{H}_*(C(k_1+\ldots+k_r, \rdbb), \qbb) \right)  ,\partial\right).
\end{multline}

\end{thm}

Note that the range is improved compared to Proposition~\ref{total_complex_homology},  because
we do not need the convergence $T_\infty C_*(\embthm, \mathbb{Q})\simeq C_*(\embthm, \mathbb{Q})$
(which does require  $\codim$). 

Since the homotopy limit does not commute with singular chains functor, in general
\[
C_*(T_n\emb,\qbb)\not\simeq T_n C_*(\emb,\qbb).
\]
Thus the methods of this paper are not applicable to understand rationally $T_n\emb$ even in the range $\codim$. 
However, we believe in the following:

\begin{conj}\label{con:trunc}
For $\codimlow$, one has
\begin{multline} \label{eq:trunc_homotopy}
\pi_*T_n\embthm \otimes \mathbb{Q} \simeq \left( \underset{0\leq k \leq n}{\bigoplus} \mathrm{hom}_{\Sigma_k} \left( \underset{1 \leq i \leq r}{\widehat{\otimes}} \overline{H}_*(C(\bullet, \rbb^{m_i}), \qbb)(k), \widehat{\pi}_*C(k, \rdbb)\otimes\qbb \right), \partial \right)\\
\simeq \left( \underset{{k_1,\ldots,k_r}\atop{k_1+\ldots +k_r\leq n}}{\bigoplus} \mathrm{hom}_{\Sigma_{k_1}\times\ldots\times\Sigma_{k_r}}
\left(\underset{1\leq i\leq r}\otimes  \overline{H}_*(C(k_i, \rbb^{m_i}),\qbb) , \widehat{\pi}_*C(k_1+\ldots+k_r,\rdbb)\otimes \qbb \right)  ,\partial\right).
\end{multline}
\end{conj}

This conjecture can be viewed as a refinement of Conjecture~\ref{con:low} and for $r=1$ has also been shown
 in~\cite[Corollary~10.13]{fr_tur_wilw}.

\begin{proof}[Proof of Theorem~\ref{th:trunc}]
Let $\Omega_n$, respectively $\Gamma_n$, denote the full subcategory of $\Omega$, respectively $\Gamma$, consisting of sets,
respectively pointed sets, of cardinal $\leq n$, respectively $\leq n+1$.  The Pirashvili cross effect functor 
restricts to an equivalence of categories also for truncated modules:
\[
\mbox{cr} \colon \underset{\Gamma_n}{\mbox{Rmod}} \lra \underset{\Omega_n}{\mbox{Rmod}}.
\]

Similarly to~\eqref{codim_condition2}, which holds for $\codimlow$, one gets
 \begin{eqnarray} \label{eq:trunc1}
  T_{n} C_*(\emb, \qbb)  \simeq 
  \underset{\Gamma_n}{\mathrm{hRmod}} (C_*((\vee_{i=1}^r S^{m_i})^{\times\bullet}), H_*(C(\bullet, \rdbb); \qbb) ). 
  \end{eqnarray}
Now using the fact that truncation commutes with the cross effect and also the formality of the $\Omega$-module
$\widetilde{C}_*((\vee_{i=1}^r S^{m_i})^{\wedge \bullet}, \qbb)$, we get
\begin{multline}\label{eq:trunc2}
T_nC_*(\embthm, \qbb) \cong \underset{\Omega_n}{\mathrm{hRmod}}\left(\widetilde{H}_*((\vee_{i=1}^r S^{m_i})^{\wedge \bullet}, \qbb), \widehat{H}_*(C(\bullet, \rdbb), \qbb)\right)\\
\simeq \underset{s_1, \cdots, s_r, t}{\prod} \underset{\Omega_n}{\mbox{hRmod}} \left(Q^{m_1 \cdots m_r}_{s_1 \cdots s_r},  \widehat{H}_{t(d-1)}(C(\bullet, \rdbb), \qbb) \right)
\simeq  \underset{s_1, \cdots, s_r, t}{\prod} \underset{\Omega}{\mbox{hRmod}} \left(Q^{m_1 \cdots m_r}_{s_1 \cdots s_r},  tr_n\widehat{H}_{t(d-1)}(C(\bullet, \rdbb), \qbb) \right),
\end{multline}
where the functor $tr_n\colon \underset{\Omega}{\mbox{Rmod}}\to \underset{\Omega}{\mbox{Rmod}}$ is defined by
\[
tr_nM(k)=
\begin{cases}
M(k),&k\leq n;\\
0,& k>n.
\end{cases}
\]
The last equality in~\eqref{eq:trunc2} follows from~\cite[Proposition~4.8]{tur_wilw2}.\footnote{This also explains
why the hairy graph-complexes can not be used to study the Goodwillie-Weiss approximations: the functor $tr_n$ does not preserve injectivity.}

Now replacing $Q^{m_1 \cdots m_r}_{s_1 \cdots s_r}$ by $CQ^{m_1 \cdots m_r}_{s_1 \cdots s_r}$, we notice that only finitely many factors in~\eqref{eq:trunc2} are non-zero, and the obtained complex is exactly the truncated~\eqref{total_complex_quasi-iso}.
\end{proof}

\begin{rmq}\label{r:multivar}
The space $\emb$ can also be studied by the multivariable manifold calculus~\cite{mun_vol12}
similarly to the case of classical (all $m_i=1$) string links considered in~\cite{mun_vol14}. We expect that analogously to Theorem~\ref{th:trunc}, one can show
that for $\codimlow$,
\begin{multline}\label{eq:trunc_homology_multi}
T_{n_1\ldots n_r}C_*(\embthm, \mathbb{Q})
\simeq \\
\left( \underset{k_i\leq n_i,\, i=1\ldots r}{\bigoplus} \mathrm{hom}_{\Sigma_{k_1}\times\ldots\times\Sigma_{k_r}}
\left(\underset{1\leq i\leq r}\otimes  \overline{H}_*(C(k_i, \rbb^{m_i}) , \widehat{H}_*(C(k_1+\ldots+k_r, \rdbb), \qbb) \right)  ,\partial\right).
\end{multline}
We also conjecture that in the same range,
\begin{multline} \label{eq:trunc_homotopy_multi}
\pi_*T_{n_1\ldots n_r}\embthm \otimes \mathbb{Q} \simeq \\
\simeq \left( \underset{k_i\leq n_i,\, i=1\ldots r}{\bigoplus} \mathrm{hom}_{\Sigma_{k_1}\times\ldots\times\Sigma_{k_r}}
\left(\underset{1\leq i\leq r}\otimes  \overline{H}_*(C(k_i, \rbb^{m_i}) , \widehat{\pi}_*C(k_1+\ldots+k_r,\rdbb)\otimes \qbb \right)  ,\partial\right).
\end{multline}
\end{rmq}

%

\textsf{Department of Mathematics and Statistics, University of Regina}\\
 3737 Wascana Pkwy, Regina, SK S4S 0A2, Canada\\
\textit{E-mail address: pso748@uregina.ca}

\textsf{Department of Mathematics, Kansas State University\\
        138 Cardwell Hall, Manhattan, KS 66506, USA \\}
\textit{E-mail address: turchin@ksu.edu}

\end{document}